\documentclass[11pt]{article}       
\usepackage{geometry}               
\usepackage{graphicx}
\usepackage{caption}
\usepackage{subcaption}
\usepackage{comment}
\usepackage{color}
\usepackage{amsmath} 
\usepackage{amssymb}  
\usepackage{amsthm}  
\usepackage{bm}
\usepackage{lipsum}
\usepackage{enumerate}
\usepackage{color}

\newtheorem{proposition}{Proposition}
\newtheorem{definition}{Definition}
\newtheorem{corollary}{Corollary}

\newtheorem{theorem}{Theorem}
\newtheorem{remark}{Remark}
\newtheorem{assumption}{Assumption}
\newtheorem{example}{Example}

\numberwithin{equation}{section}

\usepackage{lipsum}
\usepackage{marginnote,zref-savepos}
\newcounter{labelnote}
\makeatletter
\let\oldmarginnote\marginnote
\renewcommand*{\marginnote}[1]{%
 \begingroup\strut
  \stepcounter{labelnote}\zsaveposx {marginnote-\thelabelnote}
     \ifnum 0\zposx{marginnote-\thelabelnote}<19000000
      \reversemarginpar
      \oldmarginnote{{}#1}%
     \else
      \normalmarginpar
      \oldmarginnote{{}#1}%
     \fi
 \endgroup%
}
\makeatother

\title{Optimal Control of
 Conditional Value-at-Risk 
\\in Continuous Time\thanks{We would like to thank Prof. Lawrence C. Evans for useful discussions on the uniqueness of viscosity solutions.}
 } 

\author{
Christopher W. Miller\thanks{Department of Mathematics, University of California, Berkeley
({miller@math.berkeley.edu}). Supported in part by NSF GRFP under grant number DGE 1106400.}
\and
 Insoon Yang\thanks{Ming Hsieh Department of Electrical Engineering, University of Southern California ({insoonya@usc.edu}). Supported in part by NSF under CPS:FORCES (CNS1239166).}
}

\date{}
\providecommand{\keywords}[1]{\textbf{Key words.} #1}

\begin{document}
\maketitle


\pagestyle{myheadings}
\thispagestyle{plain}

\begin{abstract}
We consider continuous-time stochastic optimal control problems featuring Conditional Value-at-Risk (CVaR) in the objective. 
The major difficulty in these problems arises from time-inconsistency, which prevents us from directly using dynamic programming.
To resolve this challenge, 
we convert to an equivalent bilevel optimization problem in which the inner optimization problem is 
standard stochastic control.
Furthermore, we provide conditions under which the outer objective function is convex and differentiable. 
We compute the outer objective's value via a Hamilton-Jacobi-Bellman equation and its gradient via the viscosity solution of a linear parabolic equation, which allows us to perform gradient descent.
The significance of this result is that we provide an efficient dynamic programming-based algorithm for optimal control of CVaR without lifting the state-space. 
To broaden the applicability of the proposed algorithm, 
we propose convergent approximation schemes in cases where our key assumptions do not hold and characterize relevant suboptimality bounds.
In addition, we extend our method to a more general class of risk metrics, which includes mean-variance and median-deviation.
We also demonstrate a concrete application to portfolio optimization under CVaR constraints. Our results contribute  an efficient framework for solving time-inconsistent CVaR-based sequential optimization.
\end{abstract}

\markboth{}{Optimal Control of Conditional Value-at-Risk}

\section{Introduction}

Conditional value-at-risk (CVaR) has received significant attention 
over the past two decades
as a tool for managing risk.
CVaR measures the expected value conditional upon being within some percentage of the worst-case loss scenarios.
More specifically, the CVaR of a random variable $X$, whose distribution has no probability atoms, is defined as
\begin{equation}\nonumber
\mbox{CVaR}_\alpha(X) := \mathbb{E} \left [ X \: | \: X\geq \mbox{VaR}_\alpha (X)\right], \quad \alpha \in (0,1),
\end{equation}
where the value-at-risk (VaR) of $X$ (with the cumulative distribution function $F_X$)  is given by
\begin{equation}\nonumber
\mbox{VaR}_\alpha(X) := \inf \{ x \in \mathbb{R} \: | \: F_X(x) \geq \alpha \}.
\end{equation}
In other words, VaR is equal to $(1-\alpha)$ worst-case quantile of a loss distribution, while CVaR equals the conditional expectation of the loss within that quantile.
When the distribution has a probability atom, the definition of CVaR should be further refined (see~\cite{Rockafellar2002a}).
Note that both functions penalize only when ``bad events" occur.
 
While both VaR and CVaR are risk measures, only CVaR is \emph{coherent} in the sense of Artzner et al.~\cite{Artzner1999}.
 In addition, CVaR takes into account the possibility of tail events where losses exceed VaR.
In fact, one common criticism of VaR stems from its incapability of distinguishing  situations beyond VaR \cite{Rockafellar2002a}.
Due to the superior mathematical properties and practical implications,
CVaR has gained  popularity in risk management.\footnote{
More detailed comparisons between VaR and CVaR, in terms of
stability of statistical estimation and simplicity of optimization procedures, 
can be found in~\cite{Sarykalin2008}.}
In particular, \emph{static} or single-stage optimization with CVaR functions can be efficiently performed via convex and linear programming methods \cite{Rockafellar2000, Mansini2007}.
With the advances in optimization algorithms for CVaR,
this risk measure has shown to be useful in various finance and engineering applications. 

\emph{Dynamic} or sequential optimization of CVaR is often of interest when decisions can be made at multiple stages.
In such an optimal control setting, 
we can optimize a control action at a certain time  based on the information from observations up to that time.
This dynamic control approach enjoys an effective usage of information gathered in the process of making decisions under uncertainty.
The need for 
efficient optimal control tools with CVaR 
is also motivated by emerging dynamic risk management problems in   engineering and finance (e.g., \cite{Qin2013, Yang2017}).

The major challenge in optimal control involving CVaR arises from its \emph{time-inconsistency}  \cite{Artzner2007}.
For example, an optimal strategy for tomorrow constructed today is no longer optimal when considered tomorrow because CVaR is not a time-consistent risk measure.
Mathematically, this time-inconsistency prevents us from directly applying dynamic programming, in contrast with problems involving Markov risk measures \cite{Ruszczynski2010, Cavus2014, Ruszczynski2015} or risk-sensitive criteria \cite{James1994, Fleming1995}.
To overcome this difficulty, several methods have been proposed.
A state-space lifting approach for dynamic programming with a discrete-time and discrete-state Markov decision process (MDP) setting is first proposed in \cite{Bauerle2011}.\footnote{This lifting approach is recently generalized to semi-Markov decision processes in \cite{Huang2016} and is simplified in the application of manufacturing systems \cite{AhmadiJavid2015}.
Note that our method is different in the sense that we avoid extending the state-space by solving an associated outer optimization problem via gradient descent and therefore  reduce computational complexity in general.
} 
Another lifting method and relevant algorithms are developed in \cite{Pflug2016, Chow2015}, relying on a so-called CVaR decomposition theorem \cite{Pflug2016}.
This approach uses a dual representation of CVaR and hence requires optimization over a space of probability densities when solving an associated Bellman equation.
This optimization problem can be effectively solved in discrete-time and finite discrete-state MDPs.
However, it becomes computationally intractable in (uncountable) continuous-state optimal control problems as the space of densities is infinite dimensional.
In \cite{Haskell2015},
a different approach is developed for risk-aware discrete-time finite-state MDPs, which is based on occupation measures. 
Due to the nonconvexity of the resulting infinite-dimensional optimization problem, this method uses a successive linear approximation procedure.

In this paper, we propose a new method to solve continuous-time and continuous-space optimal control problems involving CVaR.
By using a so-called \emph{extremal} representation of CVaR originally proposed in \cite{Rockafellar2000}, we reformulate the optimal control problem as a bilevel optimization problem in which the outer optimization problem is convex and the inner optimization problem is standard stochastic optimal control.
To avoid lifting state-space, we develop a gradient descent-based method to solve the outer optimization problem.
Specifically, we prove the differentiability of the outer objective function and provide a probabilistic interpretation of its gradient under a certain semiconcave approximation.

To develop a computationally efficient and stable gradient descent-based method, it is essential to be able to compute the objective's value and gradient.
The outer objective value can be computed by solving the inner problem: we demonstrate a dynamic programming or equivalently a Hamilton-Jacobi-Bellman (HJB) method to solve the inner problem. 
More importantly, we show that the gradient of the outer objective function can be obtained as the viscosity solution of an associated linear parabolic equation, which we call the \emph{gradient} partial differential equation (PDE), under certain conditions.
These two PDE characterizations complete the proposed gradient decent-based method to solve the bilevel optimization problem.

We use two important assumptions while proving the main results.  One is the semiconcavity of the outer objective function, while the other is the uniform parabolicity of the HJB equation.
We construct convergent approximation schemes which relax each of these assumptions when needed.
For the theoretical (and also practical) implication of the approximations, we provide bounds on the gap between the optimal outer objective value and that of the perturbed problem.

In the final section of the paper, we demonstrate a practical implementation of our methodology in an optimal investment problem subject to CVaR constraints. To our knowledge, this is the first solution of a dynamic portfolio optimization problem subject to tail-risk constraints in continuous time. {{}The closest comparisons to our results are given by approximate equilibrium solutions \cite{DongSircar2014}, mean-field control approaches \cite{Pfeiffer2016}, or in mean-variance frameworks \cite{PedersenPeskir2013}.}

The rest of this paper is organized as follows:
In Section~\ref{setup}, we introduce optimal control problems involving a class of risk metrics including CVaR.
We present the main results of this paper in Section~\ref{Section:MainResults}, which are used to construct a gradient descent-based method to solve the reformulated bilevel optimization problem.
Relevant assumptions imposed in the main results are explicitly relaxed in Section~\ref{Section:Approximation} using convergent approximation schemes.
Lastly, we demonstrate the performance of the proposed method through an example of mean-CVaR portfolio selection in Section~\ref{Section:Example}.

\section{Problem Setup} \label{setup}

\subsection{Controlled Process}

Let $(\Omega,\mathcal{F},\mathbb{P})$ be a probability space supporting a standard $d$-dimensional Brownian motion $W$ with an associated filtration $\{\mathcal{F}_t\}_{0\leq t\leq T}$ satisfying the usual conditions. Let $\mathbb{A}$ be a compact and finite-dimensional set of controls.  Define the set $\mathcal{A}$ of admissible control strategies
as the collection of all $\mathcal{F}_t$-progressively measurable processes which are valued in $\mathbb{A}$ almost surely.

The control $A \in \mathcal{A}$ affects a system state of interest through the following stochastic differential equation:
\begin{equation}\label{sde}
\begin{split}
dX^A_t &= \mu(X^A_t,A_t)\,dt + \sigma(X^A_t,A_t)\,dW_t\\
X^A_0 &= \bm{x}_0 \in \mathbb{R}^n.
\end{split}
\end{equation}
We assume that $\mu:\mathbb{R}^n\times\mathbb{A}\to\mathbb{R}^n$ and $\sigma:\mathbb{R}^n\times\mathbb{A}\to\mathbb{R}^{n\times d}$ are continuous functions such that, for some $K>0$,
\begin{eqnarray}\nonumber
\|\mu(x,a)-\mu(x',a)\|+\|\sigma(x,a)-\sigma(x',a)\| & \leq & K \|x-x'\| \\
\|\mu(x,a)\| + \|\sigma(x,a)\| & \leq & K(1+\|x\|+\|a\|)\nonumber
\end{eqnarray}
for all $x,x'\in\mathbb{R}^n$ and for all $a\in\mathbb{A}$.
Under these conditions, for each control $A\in\mathcal{A}$,
 there exists a unique strong solution, $X^A$, of the SDE \eqref{sde}.
 
\subsection{Optimal Control with a Class of Risk Measures}\label{Subsection:ExtremalRiskMeasures}

The main goal of this paper is to provide an efficient algorithm for solving the following stochastic optimal  control problem with a non-standard objective\footnote{
Note that this formulation includes the following running cost problems:
$\inf\limits_{A\in\mathcal{A}}\rho \left (\int_0^T r(X^A_{1,t}, A_t) dt + q(X_{1,T}^A) \right)$.
For such problems, we introduce a new state, $X_{2,t} := \int_0^t r(X^A_{1,s}, A_s) ds$, and rewrite the optimization problem as
$\inf\limits_{A \in \mathcal{A}} \rho (X_{2,T}^A +  q(X_{1,T}^A)) =: \inf\limits_{A \in \mathcal{A}} \rho (g(X_{T}^A))$, where $X_t := (X_{1,t}, X_{2,t})$ and $g(x) = x_2 + q(x_1)$.
}:
\begin{equation}\label{opt}
\inf\limits_{A\in\mathcal{A}}\rho(g(X^A_T)),
\end{equation}
where $g:\mathbb{R}^n\to\mathbb{R}$ is a given convex cost function and $\rho:L^2(\Omega)\to\mathbb{R}$ is of a class of risk measures defined below.

\begin{definition}\label{Def:extremal}
A function $\rho: L^2(\Omega) \to \mathbb{R}$ is said to be an \emph{extremal risk measure} if there exists a convex function $f:\mathbb{R}\times\mathbb{R}^m\to\mathbb{R}$ such that
\begin{equation} \nonumber
\rho(\xi) = \inf\limits_{y\in\mathbb{R}^m}\mathbb{E}\left[f(\xi,y)\right].
\end{equation}
\end{definition}

For simplicity, we assume that $f$ is convex with at most quadratic growth and that $g$ is convex and Lipschitz continuous. However, we note that these growth conditions can easily be relaxed on a case-by-case basis by noting that $X^A_T\in L^p(\Omega)$ for all $p<\infty$.

The primary motivation of this definition is the following \emph{extremal} formula involving conditional value-at-risk (CVaR) of a random variable $\xi$, whose distribution has no probability atom\footnote{
In our problem setting, if $g(X_T^A)$ has an atom, \eqref{Eqn:ExtremalForm:CVaR} is interpreted as CVaR$^{-}$, called the lower CVaR \cite{Rockafellar2000}.
}\cite{Rockafellar2000}:
\begin{equation}\label{Eqn:ExtremalForm:CVaR}
\text{CVaR}_\alpha\left[\xi\right] = \inf\limits_{y\in\mathbb{R}}\mathbb{E}\left[y + \frac{1}{1-\alpha}\left(\xi-y\right)^+\right], \quad \alpha \in (0,1).
\end{equation}
The intuition behind this equality is that the optimal $y$ is equal to VaR at probability $\alpha$. Then, CVaR is equal to VaR plus the expected losses exceeding VaR divided by the probability of these losses occurring, $1-\alpha$.

We show that several additional risk metrics of interest in application can be written in this form. Some of these are not (coherent) risk measures in the sense of Artzner et al.~\cite{Artzner1999}. However, we justify our nomenclature of ``extremal risk measure'' by providing simple conditions on $f$ under which $\rho$ is a coherent risk measure in Appendix~\ref{app:1}. We emphasize that for ease of exposition, we will generally refer to $\rho$ as an extremal risk measure in this paper even when these conditions are not satisfied.

\begin{example}
The variance of $\xi\in L^2(\Omega)$ can be expressed as
\begin{equation}\label{Eqn:ExtremalForm:Var}
\text{Var}\left[\xi\right] = \inf\limits_{y\in\mathbb{R}}\mathbb{E}\left[\left(\xi-y\right)^2\right].
\end{equation}
Note that, with $f(\xi,y) := (\xi-y)^2$, variance is of the form of Definition~\ref{Def:extremal}. Although variance is not a risk measure, our method can handle control problems with variance-related criteria including mean-variance optimal control: $\mathbb{E}\left[\xi\right] + \lambda \text{Var}\left[\xi\right] = \inf\limits_{y\in\mathbb{R}}\mathbb{E}\left[ \xi + \lambda(\xi-y)^2\right]$.

Similarly, the median absolute deviation (MAD) of $\xi\in L^2(\Omega)$ can be expressed as
\begin{equation}\label{Eqn:ExtremalForm:MAD}
\text{MAD}\left[\xi\right] = \mathbb{E}\left[|\xi-\text{Med}\left[\xi\right]|\right] = \inf\limits_{y\in\mathbb{R}}\mathbb{E}\left[|\xi-y|\right].
\end{equation}
Again, with $f(\xi,y) := |\xi-y|$, MAD is of the form of Definition~\ref{Def:extremal}. MAD is  a deviation risk measure (e.g., \cite{Ogryczak1999}).
\end{example}

\begin{example}
Our method can handle control problems with criteria including various combinations of~\eqref{Eqn:ExtremalForm:CVaR}, \eqref{Eqn:ExtremalForm:Var}, and \eqref{Eqn:ExtremalForm:MAD}. For example,
\begin{itemize}
\item Mean-CVaR: 
\begin{equation}\nonumber
\mathbb{E}\left[\xi\right]  + \lambda \text{CVaR}_\alpha[\xi] = \inf\limits_{y\in\mathbb{R}}\mathbb{E}\left[\xi + \lambda\left(y + \frac{1}{1-\alpha} \left(\xi-y\right)^+ \right)\right];
\end{equation}
\item Variance-CVaR Tradeoff:
\begin{equation}\nonumber
\text{Var}\left[\xi\right]+\lambda\text{CVaR}_\alpha\left[\xi\right] = \inf\limits_{y\in\mathbb{R}^2}\mathbb{E}\left[\left(\xi-y_1\right)^2+\lambda\left(y_2+\frac{1}{1-\alpha}\left(\xi-y_2\right)^+\right)\right];
\end{equation}
\item $\text{CVaR}_\alpha\text{-CVaR}_\beta$ Tradeoff: 
\begin{equation}\nonumber
\text{CVaR}_\alpha\left[\xi\right]+\lambda\text{CVaR}_\beta\left[\xi\right] = \inf\limits_{y\in\mathbb{R}^2}\mathbb{E}\left[\left(\begin{array}{c}1\\\lambda\end{array}\right)\cdot\left(\begin{array}{c}y_1+(1-\alpha)^{-1}\left(\xi-y_1\right)^+\\ y_2+(1-\beta)^{-1}\left(\xi-y_2\right)^+\end{array}\right)\right].
\end{equation}
\end{itemize}
We will provide a more concrete example in Section~\ref{Section:Example}.
\end{example}

One important feature of the optimization problem \eqref{opt}  is that, in general, it is a time-inconsistent  problem. 
This makes it impossible to directly apply dynamic programming because there is not an analogue of the law of iterated expectations for non-Markov risk measures.
To tackle this issue,
one stream of research efforts focuses on developing \emph{dynamic risk measures} with which sequential optimization or optimal control can be performed in a time-consistent manner (e.g., \cite{Riedel2004, Cheridito2006, Fritelli2006, Kloppel2007, Artzner2007, Ruszczynski2006, Ruszczynski2010}).
Another class of research activities acknowledges time-inconsistency as an inherent property and proposes two different solution concepts.
The first approach is to reinterpret the problem as a game against one's future self. This approach is used, for example, in \cite{Bjork2010} and leads to a PDE system. 
The second method is to re-write the original problem in a form where we can apply dynamic programming in an indirect way. 
This approach has been used to reduce to dynamic programming in a higher-dimensional state space or to a sequence of iterated standard control problems in \cite{Bauerle2011, Miller2015, Karnam2016, Bayraktar2016, Pflug2016, Yang2017}. 
The rationale  in this paper is similar in spirit to this so-called \emph{indirect} dynamic programming method.
However, one key advantage of the proposed approach is that the structure of the extremal risk measures allows us to perform optimization over an extra variable where the objective function can be evaluated by dynamic programming involving no additional state variables. We construct optimal time-inconsistent controls by solving an equivalent bilevel optimization problem, without lifting the state-space.

\subsection{Outline of Main Results}

We summarize the main results of this paper by the following:
\begin{enumerate}
\item We demonstrate an equivalence between time-inconsistent stochastic control problem involving extremal risk measures and bilevel optimization problem in Proposition~\ref{Thm:BilevelOptimization}.
\item We provide conditions under which the outer optimization problem is convex in Theorem~\ref{Thm:Convexity}. Furthermore, under additional conditions, we prove differentiability and provide a probabilistic interpretation of the gradient in Theorem~\ref{Thm:GradientCharacterization}.
\item We demonstrate a dynamic programming approach to solving the inner problem and provide conditions under which we have a PDE characterization of the gradient of the outer problem in Theorem~\ref{Thm:GradientPDE}. This allows the use of gradient descent in solving the bilevel optimization problem.
\item Finally, we provide convergent approximations which relax two of the key assumptions imposed on the problem. In Theorems~\ref{Thm:bound2} and \ref{thm:bound1}, we provide suboptimality bounds under each approximation scheme.
\end{enumerate}

Using these results, we can solve time-inconsistent optimal control problems involving extremal risk measures via a gradient descent solution of a bilevel optimization problem. As an example of how the approximation schemes and gradient descent are utilized, consider the explicit problem:
\begin{equation} \nonumber
\inf\limits_{A\in\mathcal{A}}\text{CVaR}_\alpha\left[X^A_T\right].
\end{equation}
We can solve this problem via the following procedure:
\begin{itemize}
\item Step 1: Recall $f^{\mbox{\tiny CVaR}}(\xi, y) := y + (1-\alpha)^{-1}{\left(\xi-y\right)^+}$. This does not satisfy our uniform semiconcavity assumption and hence we first apply inf-convolution to have
\begin{equation}\nonumber
f_\epsilon (\xi, y) := \inf\limits_{z\in\mathbb{R}}\left[f^{\mbox{\tiny CVaR}}(\xi,z) + \frac{\left(y-z\right)^2}{2\epsilon}\right],
\end{equation}
which is uniformly semiconcave in $y$ ($\epsilon$-semiconcave approximation). 

\item Step 2: We note that we can reformulate the perturbed optimal control problem involving $f_\epsilon$ as the following bilevel optimization problem:
\begin{equation}\nonumber
\inf\limits_{A\in\mathcal{A}}\inf\limits_{y\in\mathbb{R}}\mathbb{E}\left[f_\epsilon(X^A_T,y)\right] = \inf\limits_{y\in\mathbb{R}}\left[V_\epsilon(y) := \inf\limits_{A\in\mathcal{A}}\mathbb{E}\left[f_\epsilon(X^A_T,y)\right]\right],
\end{equation}
where the inner problem is a standard stochastic control problem.

\item Step 3: We solve the outer problem using a gradient descent algorithm in $y$. At each iteration with $y \in \mathbb{R}^m$, we solve the inner problem via dynamic programming by solving an HJB equation to evaluate $V(y)$. In addition, we compute its gradient value $DV(y)$ by solving a linear parabolic equation, which we call the \emph{gradient PDE}.

\item Step 4: When the gradient descent algorithm converges,
we obtain a minimizer $y_\epsilon^\star$ of the outer problem and its associated minimizer $A_\epsilon^\star$ of the inner problem.
We quantify a suboptimality bound of this solution using Theorem~\ref{Thm:bound2}.
The gap between this solution and an optimal solution tends to zero as $\epsilon \to 0$.

\end{itemize}

In this high-level description of the proposed method,
we notice its several advantages.
First of all, the proposed gradient descent approach does not require us to lift the state-space in the case of terminal cost problems.
This is a considerable advantage over existing methods which regard $y$ as another state variable because the computational complexity of dynamic programming increases exponentially with the dimension of state space. Secondly, the gradient PDE provides a systematic numerical approximation of $DV(y)$ at any $y \in \mathbb{R}^m$.
Therefore, it avoids using derivative-free optimization algorithms which are not convergent in general or
automatic differentiation tools such as 
finite differencing which can be inaccurate when the objective function contains (numerical) noise
 \cite{Nocedal2006}.
 Lastly, the analytical studies to relax the assumptions used to prove the main results
broaden the applicability of the proposed method.
In particular, we show that the approximate solution obtained from an $\epsilon$-semiconcave approximation converges to the true optimal solution as $\epsilon$ tends to $0$.

{{} We summarize the complete series of approximations by stating the following general version of our results. We emphasize that significantly stronger results are provided throughout the paper when allowing various additional assumptions.
We break the statement of the following theorem into three pieces. These correspond to (1) re-writing the optimal control as a bi-level problem, (2) introducing an approximation scheme where proximal supergradients of the top-level problem can be computed in terms of a system of PDEs, and (3) a convergence result for the approximation scheme.

\begin{theorem}\label{Theorem:Summary}
\begin{enumerate}[(1)]
\item Let $f:\mathbb{R}\times\mathbb{R}^m\to\mathbb{R}$ and $g:\mathbb{R}^n\to\mathbb{R}$ both be convex and Lipschitz continuous. Define an extremal risk measure $\rho:L^2(\Omega)\to\mathbb{R}$ as
\[\rho(\xi) := \inf\limits_{y\in\mathbb{R}^m}\mathbb{E}\left[f(\xi,y)\right].\]
Furthermore, define
\[V(y) := \inf\limits_{A\in\mathcal{A}} \mathbb{E}\left[f\left(g\left(X^A_T\right),y\right)\right].\]
Then
\[p^\star := \inf\limits_{A\in\mathcal{A}}\rho(g(X^A_T)) = \inf_{y\in\mathbb{R}^m} V(y).\]

\item For any $\epsilon,\eta>0$, define
\[f_\epsilon(x,y) := \inf\limits_{z\in\mathbb{R}^m}\left[f(x,z)+\frac{\|y-z\|^2}{2\epsilon}\right].\]
Let $\hat{W}$ be an independent $n$-dimensional Brownian motion, and let $\hat{\mathcal{F}}$ be the filtration generated by the joint process $(W,\hat{W})$. Let $\hat{\mathcal{A}}$ represent the collection of all $\hat{\mathcal{F}}$-adapted processes valued in $\mathbb{A}$. For each control $A\in\hat{\mathcal{A}}$, define $\hat{X}^{A,\eta}$ as the unique strong solution of the perturbed SDE
\[d\hat{X}^{A,\eta}_t = \mu\left(\hat{X}^{A,\eta}_t,A_t\right)\,dt + \sigma\left(\hat{X}^{A,\eta}_t,A_t\right)\,dW_t+\eta\,d\hat{W}_t.\]
Define
\[V_{\epsilon,\eta}(y) := \inf\limits_{A\in\hat{\mathcal{A}}}\mathbb{E}\left[f_\epsilon(g(\hat{X}^{A,\eta}_T),y)\right]\]
and denote by $\partial V_{\epsilon,\eta}(y)$ the collection of proximal supergradients of $V_{\epsilon,\eta}$ at $y$.

Then we have the following results:
\begin{enumerate}
\item For any $y\in\mathbb{R}^m$, if we define
\[DV_{\epsilon,\eta}(y) := \mathbb{E}\left[D_y f_\epsilon(g(\hat{X}^{A,\eta}_T),y)\right],\]
then $DV_{\epsilon,\eta}(y) \in \partial V_{\epsilon,\eta}(y)$.

\item For any $y\in\mathbb{R}^m$, $V_{\epsilon,\eta}(y)$ and $DV_{\epsilon,\eta}(y)$ may be computed in terms of the solutions of two partial differential equations. The value function is related to an HJB equation, while the proximal supergradient is related to a formal linearization of the HJB equation.
\end{enumerate}

\item Let $y^\star$ be a minimizer of $V$ and $y^\star_{\epsilon,\eta}$ be a minimizer of $V_{\epsilon,\eta}$. Then
\[V(y^\star_{\epsilon,\eta})\leq V(y^\star) + C\left(\epsilon+\eta\right),\]
where $C>0$ is a constant depending only on the Lipschitz constants of $f$, $g$, $\mu$, and $\sigma$. That is, $V(y^\star_{\epsilon,\eta}) \to p^\star$ linearly as $\epsilon,\eta\to 0$.
\end{enumerate}
\end{theorem}

To summarize, this theorem provides a complete convergent algorithm for solving the optimal control of a $\rho(g(X^A_T))$. First, we show how to re-write this as a bilevel minimization problem, the lower-level of which is a standard optimal stochastic control problem. There is in general no special structure in the upper-level problem (e.g., convexity), so we demonstrate a very general semi-concave approximation\footnote{In many cases, such as the main example of this paper, we can in fact show that the approximation is a convex minimization problem, which allows more specialized techniques to be applied.}, the objective function and proximal supergradient of which can be computed in terms of solutions of a PDE system. Lastly, we demonstrate how a solution of the approximate problem converges to an optimal solution of the original problem.

We will present a proof of Theorem~\ref{Theorem:Summary} in Section~\ref{Subsection:ProofMainTheorem} after developing the relevant machinery and approximations in Section~\ref{Section:MainResults} and Section~\ref{semi_app}--\ref{Subsection:Nonconcave}.
}

\section{Main Results: Gradient Descent and Viscosity Solutions}\label{Section:MainResults}

\subsection{Equivalent Bilevel Optimization}

Recall that our goal is to solve the generalized stochastic control problem
\begin{equation}\label{Eqn:GeneralizedOptimization}
\inf\limits_{A\in\mathcal{A}}\rho(g(X^A_T)),
\end{equation}
where $\rho$ is a fixed extremal risk measure and $g(X^A_T)$ is the state-dependent cost when the control $A$ is executed.

In general, \eqref{Eqn:GeneralizedOptimization} is a time-inconsistent nonlinear stochastic optimal control problem to which we cannot apply dynamic programming. However, we show how to use the structure of $\rho$ as an extremal risk measure to convert this into an equivalent bilevel optimization problem.

{{} \begin{proposition}[Bilevel Optimization]\label{Thm:BilevelOptimization}
We can write the problem of dynamic optimization over an extremal risk measure as
\begin{equation}\label{Eqn:BilevelOptimization}
\inf\limits_{A\in\mathcal{A}}\rho\left(g(X^A_T)\right) = \inf\limits_{y\in\mathbb{R}^m}V(y),
\end{equation}
where $V$ is defined via a standard stochastic optimal control problem of the form
\begin{equation}\label{Eqn:StandardControl}
V(y) := \inf\limits_{A\in\mathcal{A}}\mathbb{E}\left[f(g(X^A_T),y)\right].
\end{equation}
\end{proposition}}
The proof is straightforward by the definition of extremal risk measures.
\begin{remark}
The value $V(y)$ depends on the initial value $\bm{x}_0$ of $X^A$.
For  simplicity, however, we suppress the dependency and implicitly assume that we fix the initial value as $\bm{x}_0$ for the rest of this paper.
\end{remark}

At this point, we have converted the time-inconsistent stochastic control problem \eqref{Eqn:GeneralizedOptimization} to an equivalent bilevel optimization problem \eqref{Eqn:BilevelOptimization} involving a standard stochastic control problem \eqref{Eqn:StandardControl}. 
For convenience, we call the right-hand sides of \eqref{Eqn:BilevelOptimization} and \eqref{Eqn:StandardControl}
 the \emph{outer optimization} problem and the \emph{inner optimization} problem, respectively.

\subsection{A Note on Assumptions}
We record three main assumptions. These are used to obtain various properties of the bilevel optimization problem along the way, and we discuss where applicable they may be relaxed in Section~\ref{Section:Approximation}.

\begin{assumption}[Uniform Semiconcavity]\label{Assumption:UniformSemiconcavity}
The function $y\mapsto f(x,y)$ is uniformly  semiconcave for all $x\in\mathbb{R}$. That is, there exists $M>0$ such that
\begin{equation}\nonumber
f(x,y+\xi) \leq f(x,y) + \xi\cdot D_y f(x,y) + \frac{1}{2}M\|\xi\|^2
\end{equation}
for all $(x,y)\in\mathbb{R}\times\mathbb{R}^m$ and for all $\xi\in\mathbb{R}^m$. 
\end{assumption}

It is important to note that the convexity of $f$ together with Assumption~\ref{Assumption:UniformSemiconcavity} guarantees that the map $y\mapsto f(x,y)$ is continuously differentiable for each $x\in\mathbb{R}^m$. This regularity will ultimately carry over to various smoothness results of the outer optimization problem in our bilevel optimization problem, i.e., the right-hand side of \eqref{Eqn:BilevelOptimization}. 
\begin{remark}
Relaxing Assumption~\ref{Assumption:UniformSemiconcavity} is particularly important from the application perspective.
Note, for example, that CVaR and MAD do not satisfy this uniform semiconcavity assumption. 
For the clarity of presentation, however, we begin with this seemingly restrictive assumption but later relax it by using the inf-convolution operator in Section~\ref{semi_app}.
\end{remark}

\begin{assumption}[Uniform Parabolicity]\label{Assumption:UniformEllipticity}
There exists an $\epsilon >0$ such that
\begin{equation}\nonumber
\sigma(x,a)\sigma(x,a)^\top - \epsilon I_n\text{ is positive semidefinite}
\end{equation}
for all $(x,a)\in\mathbb{R}^n\times\mathbb{A}$.
\end{assumption}

The main use of Assumption~\ref{Assumption:UniformEllipticity} is to construct optimal controls to the stochastic control problem~\eqref{Eqn:StandardControl}. In particular, this assumption guarantees that the viscosity solution to the Hamilton-Jacobi-Bellman equation is smooth, \cite{Fleming2006,Lions1983}. In Section~\ref{uni_app}, we relax this constraint by adding additional independent Brownian motions to the dynamics of $X^A$.

\begin{assumption}[Convexity]\label{Assumption:Convexity}
The control set $\mathbb{A}$ is convex and the map $(A,y) \mapsto f(g(X^A_T),y)$ is jointly convex, almost surely.
\end{assumption}

The primary use of Assumption~\ref{Assumption:Convexity} is to guarantee convexity of the outer optimization problem in our bilevel optimization formulation. In particular, this allows us to implement a gradient descent algorithm with guaranteed convergence to a global minimum. When Assumption~\ref{Assumption:Convexity} does not hold, we can still compute so-called proximal supergradients and run a descent algorithm which converges to a local minimizer.

\begin{example}\label{Example:SufficientConditionsConvexity}
We emphasize the following three sufficient conditions, each of which guarantees Assumption~\ref{Assumption:Convexity} holds (e.g., \cite{Boyd2004}):
\begin{itemize}
\item $A\mapsto g(X^A_T)$ is affine,
\item  $A\mapsto g(X^A_T)$ is convex\footnote{We acknowledge that checking the convexity of $A \mapsto X_T^A$ beyond the criteria proposed in \cite{Azhmyakov2008, Colaneri2014} is often a nontrivial task. 
It is  a  topic of future research.} and $x\mapsto f(x,y)$ is non-decreasing convex for each $y\in\mathbb{R}^m$, or
\item $A\mapsto g(X^A_T)$ is concave and $x\mapsto f(x,y)$ is non-increasing convex for each $y\in\mathbb{R}^m$.
\end{itemize}
Recall that $f^{\mbox{\tiny CVaR}}(\xi, y) := y + (1-\alpha)^{-1}{\left(\xi-y\right)^+}$, $x \mapsto f^{\mbox{\tiny CVaR}}(x, y)$ is non-decreasing convex.
Therefore, if $A \mapsto g(X_T^A)$ is convex, Assumption~\ref{Assumption:Convexity} holds.
\end{example}

{{} We note that Assumption~\ref{Assumption:Convexity} is quite strong. However, it proves  verifiable in some practical applications in engineering and finance such as risk-aware demand response, electric vehicle charging control, inventory control and portfolio management (e.g., \cite{Rajagopal2012, Yang2015}). As a concrete example, consider the Mean-CVaR portfolio optimization problem considered in Section~\ref{Section:Example}. 
For clarity of exposition, we retain this as a main assumption but emphasize the analogous results which hold even when this assumption is broken.}

\subsection{Analytical Properties of the Outer Objective Function $V$}

In this section we investigate some analytical properties of the outer objective function $V$. 
We begin by showing the convexity of $V$.
Then, we present a semiconcavity estimate of $V$ at points where there exists an optimal control. 
Furthermore, we use this estimate and convexity to show the differentiability of $V$  at such points and to provide a probabilistic representation of the gradient.
\begin{theorem}[Convexity]\label{Thm:Convexity}
Suppose that Assumption~\ref{Assumption:Convexity} holds.
Then, the outer objective function $V$ is convex.
\end{theorem}

The convexity of $V$ motivates us to use a (sub)gradient descent-type algorithm for solving the outer optimization problem.
Under the semiconcavity assumption (Assumption~\ref{Assumption:UniformSemiconcavity}),
we justify gradient descent approaches by proving the differentiability of $V$.

To that end, we first record a semiconcavity estimate. It is important to note that this estimate does not depend on the convexity of $V$ or Assumption~\ref{Assumption:Convexity}. This result connects the semiconcavity of $f$ to that of $V$.

\begin{proposition}\label{Prop:Semiconcavity}
Suppose that Assumption~\ref{Assumption:UniformSemiconcavity}  holds.
For any fixed $y\in\mathbb{R}^m$, we assume there exists $A\in\mathcal{A}$
such that\footnote{We assume the existence of an optimal control in this proposition. 
However, we will show that this assumption is valid in the next subsection by constructing an optimal control from an associated HJB equation under Assumption~\ref{Assumption:UniformEllipticity}.} 
\begin{equation}\nonumber
V(y) = \mathbb{E}\left[f(g(X^A_T),y)\right].
\end{equation}
Then, we have
\begin{equation}\nonumber
V(y+\xi) \leq V(y) + \xi\cdot\mathbb{E}\left[D_yf(g(X^A_T),y)\right] + \frac{1}{2}M\|\xi\|^2
\end{equation}
for all $\xi\in\mathbb{R}^m$.
\end{proposition}
\begin{proof}
Fix $y\in\mathbb{R}^m$. By Assumption~\ref{Assumption:UniformSemiconcavity}, there exists $M>0$ such that
\begin{equation}\label{Eqn:SemiconcaveInequality}
f(x,y+\xi) \leq f(x,y) + \xi\cdot D_yf(x,y) + \frac{1}{2}M\|\xi\|^2
\end{equation}
for all $x\in\mathbb{R}$ and $\xi\in\mathbb{R}^m$. Let $A\in\mathcal{A}$ be a control such that
\begin{equation}\nonumber
V(y) = \mathbb{E}\left[f(g(X^A_T),y)\right].
\end{equation}
Note that such a control depends on the choice of $y$.
If we apply inequality \eqref{Eqn:SemiconcaveInequality} pointwise and take expectations, we see
\begin{equation}\nonumber
V(y+\xi) \leq \mathbb{E}\left[f(g(X^A_T),y+\xi)\right] \leq \mathbb{E}\left[f(g(X^A_T),y)\right] + \xi\cdot\mathbb{E}\left[D_y f(g(X^A_T),y)\right] + \frac{1}{2}M\|\xi\|^2
\end{equation}
for all $\xi\in\mathbb{R}^m$. Then, the result holds.
\end{proof}

In the following result, we combine the convexity and semiconcavity estimates of $V$ to show that $V$ is in fact differentiable. In particular, we provide a probabilistic representation of the gradient at each point.

\begin{theorem}[Differentiability]\label{Thm:GradientCharacterization}
Suppose that Assumptions~\ref{Assumption:UniformSemiconcavity} and~\ref{Assumption:Convexity} hold.
For any fixed $y\in\mathbb{R}^m$, we assume that there exists $A\in\mathcal{A}$ such that
\begin{equation}\nonumber
V(y) = \mathbb{E}\left[f(g(X^A_T),y)\right].
\end{equation}
Then, $V$ is differentiable at $y$ and its gradient can be computed as
\begin{equation}\label{Eqn:GradientCharacterization}
DV(y) = \mathbb{E}\left[D_y f(g(X^A_T),y)\right].
\end{equation}
\end{theorem}

{{}
We emphasize that, while the value of the gradient seems the natural consequence of an envelope theorem, we do not know that $V$ is differentiable a priori. This theorem proves differentiability as a direct consequence of convexity and semiconcavity, while identifying the gradient in the process.}

\begin{proof}
{{}
The function $V$ is convex by Theorem~\ref{Thm:Convexity}. Because of the quadratic growth of $f$, the Lipschitz assumptions on $g$, $\mu$, and $\sigma$, and the compactness of $\mathbb{A}$, it is standard result that $V(y)$ is finite for every $y\in\mathbb{R}^m$ (see \cite{Evans2013}). Therefore, its subdifferential is non-empty at each point \cite{Rockafellar1970}.} Fix $y\in\mathbb{R}^m$ and suppose $A\in\mathcal{A}$ is a control such that
\begin{equation} \nonumber
V(y) = \mathbb{E}\left[f(g(X^A_T),y)\right].
\end{equation}
Let $z\in\partial V(y)$ be an arbitrary subgradient of $V$ at $y$, i.e.,
\begin{equation}\nonumber
V(y+\xi) \geq V(y) + \xi\cdot z\hspace{0.5cm}\forall\xi\in\mathbb{R}^m.
\end{equation}
By Proposition~\ref{Prop:Semiconcavity}, we also have the inequality
\begin{equation}\nonumber
V(y+\xi) \leq V(y) + \xi\cdot\mathbb{E}\left[D_y f(g(X^A_T),y)\right]+\frac{1}{2}M\|\xi\|^2\hspace{0.5cm}\forall\xi\in\mathbb{R}^m.
\end{equation}
Putting these together, we obtain
\begin{equation}\nonumber
\xi\cdot(z-\mathbb{E}\left[D_y f(g(X^A_T),y)\right]) \leq \frac{1}{2}M\|\xi\|^2\hspace{0.5cm}\forall\xi\in\mathbb{R}^m.
\end{equation}
Choosing $\xi := M^{-1}(z-\mathbb{E}\left[D_y f(g(X^A_T),y)\right])$, we have
\begin{equation}\nonumber
M \|\xi\|^2 = \xi\cdot(z-\mathbb{E}\left[D_y f(g(X^A_T),y)\right]) \leq \frac{1}{2}M\|\xi\|^2,
\end{equation}
which is a contradiction unless $z=\mathbb{E}\left[D_y f(g(X^A_T),y)\right]$.

Because $\partial V(y)$ is single-valued, we conclude that $V$ is differentiable at $y$ and also that
\begin{equation}\nonumber
DV(y) = \mathbb{E}\left[D_y f(g(X^A_T),y)\right]
\end{equation}
as desired.
\end{proof}
Due to Theorem~\ref{Thm:GradientCharacterization}, 
we can solve the outer optimization problem using a gradient descent algorithm given that the function value $V(y)$ and its gradient $DV(y)$ are provided. 
We notice that $V(y)$  can be computed by dynamic programming. 
It is worth mentioning that the overall problem is still time-inconsistent but our bilevel decomposition allows us to solve the inner optimization problem using dynamic programming.
We investigate the inner optimization problem
and provide a constructive approach to solve the overall problem
 in the following subsection.

\subsection{PDE Characterization of $V$ and $DV$}\label{Section:PDE}

Recall that $V(y)$ is the optimal value of the inner optimization problem \eqref{Eqn:StandardControl} given the initial value $\bm{x}_0$ of the state.
Note that this problem can be solved by dynamic programming,
we first compute $V(y)$ in terms of the viscosity solution of an associated HJB equation. In the process we construct an optimal control, which guarantees $V$ is differentiable by Theorem~\ref{Thm:GradientCharacterization}. Furthermore, we can compute $DV(y)$ in terms of the viscosity solution of a linear parabolic equation.

\begin{proposition}\label{Prop:ViscositySolution}
Given $y\in\mathbb{R}^m$, let $v: [0,T] \times \mathbb{R}^n \to \mathbb{R}$ be the viscosity solution of the HJB equation
\begin{equation}\label{Eqn:ValueHJB}
\begin{array}{rl}
v_t + \inf\limits_{a\in\mathbb{A}}\left[\frac{1}{2} \mbox{tr}( \sigma(x,a)\sigma(x,a)^\top D^2_x v ) + \mu(x,a)\cdot D_x v\right] = 0 & \text{in }[0,T)\times\mathbb{R}^n\\
v(T,x) = f(g(x),y) & \text{on }\left\{t=T\right\}\times\mathbb{R}^n.
\end{array}\end{equation}
Then, we have
\begin{equation}\nonumber
V(y) = v(0,\bm{x}_0),
\end{equation}
where $\bm{x}_0$ is the initial value of the SDE \eqref{sde}.
\end{proposition}

{{}This is a standard result. The important point is that, under Assumption~\ref{Assumption:UniformEllipticity}, the equation \eqref{Eqn:ValueHJB} is uniformly parabolic and concave in $D_x^2v$. Then we know from regularity results for HJB equations that the value function $v$ is twice differentiable in space (see \cite{Evans1983,Caffarelli1995,Gilbarg2001}). This allows us to compute $DV(y)$ by solving a linear equation.}

\begin{theorem}[Gradient PDE]\label{Thm:GradientPDE}
Suppose that Assumptions~\ref{Assumption:UniformEllipticity}--\ref{Assumption:Convexity} hold.
Given $y\in\mathbb{R}^m$,
let $v$ be the viscosity solution of \eqref{Eqn:ValueHJB}. 
\begin{enumerate}
\item An optimal control, $A^\star_t := a^\star(t,X^{A^\star}_t)$, exists, where
$a^\star:[0,T)\times\mathbb{R}^n\to\mathbb{A}$ satisfies
\begin{equation}\nonumber
a^\star(t,x)\in \arg\min\limits_{a\in\mathbb{A}}\left[\frac{1}{2} \mbox{tr} ( \sigma(x,a)\sigma(x,a)^\top D^2_x v) + \mu(x,a)\cdot D_x v\right]\text{ for all }(t,x)\in[0,T)\times\mathbb{R}^n.
\end{equation}
\item
Let $w:[0,T]\times\mathbb{R}^n\to\mathbb{R}^m$ be defined as
\begin{equation}\label{Eqn:GradientViscositySolution}
w(t,x) := \mathbb{E}_t\left[D_y f(g(X^{A^\star}_T),y)\mid X^{A^\star}_t=x\right].
\end{equation}
Then, $w := (w^{(1)}, \cdots, w^{(m)})$ is a viscosity solution of the decoupled system of linear equations
\begin{equation}\label{Eqn:GradientPDE}
\begin{array}{rl}
w_t^{(k)} + \frac{1}{2}\mbox{tr}( \sigma(x,a^\star(t,x))\sigma(x,a^\star(t,x))^\top D^2_x w^{(k)})  + \mu(x,a^\star(t,x))\cdot D_x w^{(k)} = 0 & \text{in }[0,T)\times\mathbb{R}^n\\
w^{(k)}(T,x) = [D_y f(g(x),y)]_k & \text{on }\left\{t=T\right\}\times\mathbb{R}^n
\end{array}\end{equation}
for $k=1, \cdots, m$.
Furthermore, we have
\begin{equation}\nonumber
DV(y) = w(0,\bm{x}_0),
\end{equation}
where $\bm{x}_0$ is the initial value of the SDE \eqref{sde}.
\end{enumerate}
\end{theorem}

Before providing a rigorous proof, we note the intuition behind the result. It is clear that $w$ defined in (\ref{Eqn:GradientViscositySolution}) satisfies
\begin{equation}\label{Eqn:ConditionalTowerRule}
w(t,x) = \mathbb{E}_t\left[w(t+h,X^{A^*}_{t+h}) \mid X^{A^*}_t = x\right]
\end{equation}
for all $(t,x)\in[0,T)\times\mathbb{R}^n$ and $h\in(0,T-t)$. This follows from the law of iterated expectations. If $w$ were smooth, we could apply Ito's lemma for small $h>0$ and conclude that $w$ satisfies \eqref{Eqn:GradientPDE}. However,  the coefficients of \eqref{Eqn:GradientPDE} are not necessarily continuous and hence the solution $w$ is potentially discontinuous. That is, we must consider this within the framework of discontinuous viscosity solutions. For more on this topic, see~\cite{Katzourakis2015,Touzi2013}.

\begin{proof}
\begin{enumerate}
\item
The existence of such an optimal control follows from a standard argument due to the uniform parabolicity of the HJB \eqref{Eqn:ValueHJB} (e.g., \cite{Fleming2006}).

\item
Without loss of generality, we can assume $m=1$ and suppress the superscript of $w$ because the PDE system \eqref{Eqn:GradientPDE} is decoupled.
Our goal is then to show that the function $w$ defined in \eqref{Eqn:GradientViscositySolution} is a discontinuous viscosity solution of \eqref{Eqn:GradientPDE}. The key property of $w$ is listed in \eqref{Eqn:ConditionalTowerRule}. For convenience, we introduce the notation
\begin{equation} \nonumber
(\mathcal{L}^{a^\star}\phi)(t,x) := \phi_t(t,x) + \mbox{tr} (\sigma(x,a^\star(t,x))\sigma(x,a^\star(t,x))^\top D^2_x\phi(t,x)) + \mu(x,a^\star(t,x))\cdot D_x\phi(t,x),
\end{equation}
where $\phi$ is an arbitrary smooth function.

As a reminder, we define the lower- and upper-semicontinuous envelopes of a locally-bounded function $\psi$ as
\begin{equation}\nonumber
\psi_*(x) := \liminf\limits_{y\to x}\psi(y)\text { and }\psi^*(x) := \limsup\limits_{y\to x}\psi(y),
\end{equation}
respectively.

Fix $(\bar{t},\bar{x})\in[0,T)\times\mathbb{R}$ and let $\phi:[0,T)\times\mathbb{R} \to\mathbb{R}$ be a smooth function satisfying
\begin{equation}\nonumber
0 = (w_{*}-\phi)(\bar{t},\bar{x}) = \min\limits_{[0,T)\times\mathbb{R}^n}(w_{*}-\phi).
\end{equation}
That is, $\phi$ touches the lower-semicontinuous envelope of $w$ from below at $(\bar{t},\bar{x})$. Our goal is to show that
\begin{equation}\nonumber
\left(\mathcal{L}^{a^\star}\phi\right)_*(\bar{t},\bar{x}) \leq 0.
\end{equation}
Towards that end, let $\{(t_k,x_k)\}_{k=0}^\infty$ be a sequence such that, as $k \to \infty$,
\begin{equation}\nonumber
(t_k,x_k) \to (\bar{t},\bar{x})\text{ and }w(t_k,x_k)\to w_{*}(\bar{t},\bar{x}).
\end{equation}
Since $\phi$ is smooth,  $\delta_k := w(t_k,x_k) - \phi(\bar{t},\bar{x}) \to 0$. We also define
\begin{equation}\nonumber
h_k := \sqrt{\delta_k}\,1_{\left\{\delta_k \neq 0\right\}}+k^{-1}\,1_{\left\{\delta_k = 0\right\}}
\end{equation}
and take $k$ large enough that $h_k\in(0,T-t_k)$. Then, we have
\begin{eqnarray}\nonumber
w(t_k,x_k) & = & \mathbb{E}_{t_k}\left[w(t_k+h_k,X^{A^\star}_{t_k+h_k}) \mid X^{A^\star}_{t_k} = x_k\right] \\\nonumber
& \geq & \mathbb{E}_{t_k}\left[\phi(t_k+h_k,X^{A^\star}_{t_k+h_k}) \mid X^{A^\star}_{t_k} = x_k\right] \\\nonumber
& = & \phi(t_k,x_k) + \mathbb{E}_{t_k}\left[\int_{t_k}^{t_k+h_k} \left(\mathcal{L}^{a^\star}\phi\right)(s,X^{A^\star}_s)\,ds \mid X^{A^\star}_{t_k}=x_k\right].\nonumber
\end{eqnarray}
Rearranging this, we conclude
\begin{equation}\nonumber
\frac{\delta_k}{h_k} \geq \mathbb{E}_{t_k}\left[\frac{1}{h_k}\int_{t_k}^{t_k+h_k} \left(\mathcal{L}^{a^\star}\phi\right)(s,X^{A^\star}_s)\,ds \mid X^{A^\star}_{t_k}=x_k\right].
\end{equation}
Sending $k\to\infty$, the left-hand side converges to zero, while the right-hand side dominates the lower-semicontinuous envelope of $\mathcal{L}^{a^\star}\phi$ almost surely. Therefore, we conclude by the dominated convergence theorem that
\begin{equation}\nonumber
\left(\mathcal{L}^{a^\star}\phi\right)_*(\bar{t},\bar{x}) \leq 0.
\end{equation}

The opposite inequality proceeds exactly the same. We also notice that $w$ satisfies the boundary condition of \eqref{Eqn:GradientPDE} pointwise. Therefore, $w$ is a viscosity solution of the linear PDE~\eqref{Eqn:GradientPDE}.

Finally, we put this together with Theorem~\ref{Thm:GradientCharacterization} to conclude
\begin{equation}\nonumber
w(0,\bm{x}_0) = \mathbb{E}\left[D_y f(g(X^{A^\star}_T),y)\right] = DV(y).
\end{equation}
\end{enumerate}
\end{proof}

We note that some care must be taken when arguing that the viscosity solution to the PDE in (\ref{Eqn:GradientPDE}) is unique. It is well-known that non-divergence form linear PDEs may not have a unique viscosity solution in dimensions greater than two if the diffusion coefficients are not continuous (see \cite{Safonov1999,Krylov2004,Nadirashvili1997}). However, there has been subsequent work on finding structural conditions on the coefficients which guarantee weak uniqueness in higher dimensions. For fully non-linear elliptic equations, \cite{Kawohl1998} provides assumptions on the nonlinearity, not including continuous coefficients, which guarantee unique upper semi-continuous viscosity solutions in the sense of Theorem~\ref{Thm:GradientPDE}. Recent results focus on bounds on the mean-oscillation of the coefficients \cite{Krylov2007,Kim2007,Dong2015}, as well as on the set of discontinuities being small \cite{Soravia2006,Kim2007_2}.

We emphasize that in many practical problems
the linear parabolic equation \eqref{Eqn:GradientPDE} has a unique viscosity solution.
In the next proposition, we provide two conditions which can be easily checked. 
\begin{proposition}
Define $F:\mathbb{R}^n\times\mathbb{R}^n\times\mathbb{S}^n\to\mathbb{R}$ as
\begin{equation}\nonumber
F(x,p,R) := \inf\limits_{a\in\mathbb{A}}\left[\frac{1}{2}\mbox{tr}(\sigma(x,a)\sigma(x,a)^\top R) + \mu(x,a)\cdot p\right].
\end{equation}
If either (i) $n\leq 2$ or (ii) both $D_p F$ and $D_R F$ exist and are continuous, then there is a unique viscosity solution of (\ref{Eqn:GradientPDE}).
\end{proposition}
\begin{proof}

If $n\leq 2$, then this follows from Theorem~2.17 of \cite{Krylov2004}. If $D_p F$ and $D_R F$ exist and are continuous, then we note that the gradient PDE (\ref{Eqn:GradientPDE}) can be re-written in the form
\begin{equation}\nonumber
w_t^{(k)} + \mbox{tr} (D_R F(x,D_xu,D^2_xu)^\top D_x^2 w^{(k)} ) + D_p F(x,D_xu,D^2_xu)\cdot D_x w^{(k)} = 0.
\end{equation}
Recall $u$ has continuous second derivatives in space, so the coefficients are all continuous functions. Then, the uniqueness of viscosity solutions follows as usual.

\end{proof}

Using Proposition~\ref{Prop:ViscositySolution} and Theorem~\ref{Thm:GradientPDE}, we can calculate $V(y)$ and $DV(y)$ at each $y \in \mathbb{R}^m$ by (numerically) solving the PDEs \eqref{Eqn:ValueHJB} and \eqref{Eqn:GradientPDE}.
Therefore, we can solve the outer optimization problem using a gradient descent-type algorithm and find a globally optimal solution due to the convexity of $V$.
We will not discuss numerical optimization algorithms as they are not the focus of this paper (we refer to, for example, \cite{Nocedal2006} for detailed algorithms).

{{}One natural question regarding the computation of the gradient is whether it is possible to utilize other numerical methods, most obviously Monte Carlo methods. Theorem~\ref{Thm:GradientCharacterization} can of course be used directly for this. The practitioner would obtain a numerical solution of \eqref{Eqn:ValueHJB}, then generate random optimal trajectories via Monte Carlo simulations, and estimate $DV(y)$ via the sample expected value corresponding to \eqref{Eqn:GradientCharacterization}. This however, introduces extra numerical (sampling) error in the calculation of the gradient, which is often a source of instability.
}

Furthermore, this gradient descent approach gains a dimensionality reduction by $m$, which is the dimension of $y$, as opposed to a dynamic programming method over the lifted state space of $(x, y)$.  Even when $m = 1$, this gain is considerable because the computational complexity of dynamic programming increases exponentially with the dimension of state space.\footnote{{{}
One can further alleviate the computational burden by employing advanced numerical methods such as sparse grids and multigrid techniques (e.g., \cite{Hoppe1986, Bokanowski2013}). It is difficult to make precise statements about the complexity of the entire algorithm without discussing certain implementation choices. However, it is essentially that of a gradient-based unconstrained concave minimization algorithm, where each step requires the solution of an $(n+1)$-dimensional HJB and $m$ linear PDEs of dimension $(n+1)$.}}

\section{Relaxation of Assumptions}\label{Section:Approximation}

The goal of this section is to explicitly relax Assumptions~\ref{Assumption:UniformSemiconcavity} and~\ref{Assumption:UniformEllipticity}, and to a lesser extent Assumption~\ref{Assumption:Convexity}. In the case of Assumptions~\ref{Assumption:UniformSemiconcavity} and~\ref{Assumption:UniformEllipticity}, we provide convergent approximation schemes including suboptimality bounds.

\subsection{Uniformly Semiconcave Approximation}\label{semi_app}

In this section, we consider an approximation scheme for the case that $f$ does not satisfy the semiconcavity assumption (Assumption~\ref{Assumption:UniformSemiconcavity}). The idea is to modify $f$ via inf-convolution to obtain a function $f_\epsilon$ for some small $\epsilon>0$. We show that this new function $f_\epsilon$ satisfies Assumption~\ref{Assumption:UniformSemiconcavity}. We then prove that the resulting perturbed value function converges to the unperturbed problem as $\epsilon\to 0$.
Relaxing this semiconcavity assumption is particularly important for problems with CVaR objectives.

First, we recall the definition of inf-convolution and the key properties of the resulting function.

\begin{proposition}[Semiconcave approximation]\label{Prop:InfConvolutionProperties}
Recall that $f : \mathbb{R} \times \mathbb{R}^m \to \mathbb{R}$ is convex.
For $\epsilon >0$, define the inf-convolution $f_\epsilon$ as
\begin{equation} \label{inf_conv}
f_\epsilon(x,y) := \inf\limits_{z\in\mathbb{R}^m}\left[f(x,z) + \frac{\|y-z\|^2}{2\epsilon}\right].
\end{equation}
Then, $f_\epsilon: \mathbb{R} \times \mathbb{R}^m \to \mathbb{R}$ has the following properties:
\begin{enumerate}
\item $f_\epsilon$ is convex,
\item $y\mapsto f_\epsilon(x,y)$ is $\epsilon$-semiconcave for all $x\in\mathbb{R}^n$, and
\item $f_\epsilon\to f$ uniformly as $\epsilon\to 0$.
\end{enumerate}
\end{proposition}

For a proof of Proposition~\ref{Prop:InfConvolutionProperties}, see \cite{Katzourakis2015,Rockafellar1970}. In particular, this implies that $f_\epsilon$ satisfies Assumption~\ref{Assumption:UniformSemiconcavity}. We also show that approximation by inf-convolution will not break the convexity required by Assumption~\ref{Assumption:Convexity}.

\begin{proposition}\label{Prop:InfConvPreseveresConvexity}
If $f$ satisfies Assumption~\ref{Assumption:Convexity}, then so does $f_\epsilon$ for each $\epsilon > 0$.
\end{proposition}

\begin{proof}
See Appendix~\ref{app:Prop7}.
\end{proof}

We now define the $\epsilon$-perturbed outer objective function as
\begin{equation} \nonumber
V_\epsilon(y) := \inf\limits_{A\in\mathcal{A}}\mathbb{E}\left[f_\epsilon(g(X^A_T),y)\right].
\end{equation}
For any $\epsilon >0$, the perturbed function $f_\epsilon$ satisfies Assumption~\ref{Assumption:UniformSemiconcavity}. Then, we can apply the results of Section~\ref{Section:MainResults} to minimize $V_\epsilon$. The goal of this subsection is to show convergence as $\epsilon\to 0$.

First, we record an estimate involving inf-convolution.

\begin{proposition}\label{prop:ic_estimate}
Suppose that $y\mapsto f(x,y)$ is uniformly Lipschitz continuous for all $x\in\mathbb{R}$. Fix $\epsilon >0$ and $(x,y)\in\mathbb{R}\times\mathbb{R}^m$. Then, we have
\begin{equation}\nonumber
f_\epsilon(x,y) \leq f(x,y) \leq f_\epsilon(x,y) + C\epsilon
\end{equation}
for a constant $C$ depending only on $f$.
\end{proposition}

\begin{proof}
Taking $z=y$ in the inf-convolution operator \eqref{inf_conv}, we immediately have that
\begin{equation}\nonumber
f_\epsilon(x,y) \leq f(x,y).
\end{equation}
Next, let $z^\star \in\mathbb{R}^m$ be a minimizer in $f_\epsilon$. Using this and the uniform Lipschitz continuity of $f$, we have
\begin{equation}\nonumber
f_\epsilon(x,y) = f(x,z^\star) + \frac{\|y-z^\star\|^2}{2\epsilon} \geq f(x,y) - L\|y-z^\star\| + \frac{\|y-z^\star\|^2}{2\epsilon}
\end{equation}
for some $L>0$. On the other hand, by the Cauchy-Schwarz inequality, we have
\begin{equation}\nonumber
L\|y-z\| \leq \frac{1}{2}\epsilon L^2 + \frac{\|y-z\|^2}{2\epsilon}.
\end{equation}
Putting these inequalities together, we then obtain
\begin{equation}\nonumber
f_\epsilon(x,y) \geq f(x,y) - \frac{1}{2}\epsilon L^2.
\end{equation}
\end{proof}

Next, we prove a bound relating  the perturbed and unperturbed outer objective functions.

\begin{proposition}\label{Prop:SemiConcavePerturbedValueEstimate}
Suppose that $y\mapsto f(x,y)$ is uniformly Lipschitz continuous for all $x\in\mathbb{R}$. Fix $\epsilon >0$ and $y\in\mathbb{R}^m$. Then, we have
\begin{equation}\nonumber
V_\epsilon(y) \leq V(y) \leq V_\epsilon(y) + C\epsilon
\end{equation}
for a constant $C$ depending only on $f$.
\end{proposition}

\begin{proof}
By Proposition~\ref{prop:ic_estimate}, there exists a constant $C$ such that
\begin{equation}\nonumber
f_\epsilon(g(X^A_T),y) \leq f(g(X^A_T),y) \leq f_\epsilon(g(X^A_T),y) + C\epsilon \quad \text{almost surely},
\end{equation}
for all $A\in\mathcal{A}$. Taking expectation and using the fact that $A\in\mathcal{A}$ is arbitrary, we have
\begin{equation}\nonumber
V_\epsilon(y) \leq V_0(y) \leq V_\epsilon(y) + C\epsilon.
\end{equation}
Noting that $V_0(y)=V(y)$, the result follows.
\end{proof}

Finally, we use this result to prove a suboptimality bound when using this approximation.

\begin{theorem}[Convergence and suboptimality bound I]\label{Thm:bound2}
Suppose that $y\mapsto f(x,y)$ is uniformly Lipschitz continuous for all $x\in\mathbb{R}$. Fix $\epsilon >0$ and let $y^\star_\epsilon\in\mathbb{R}^m$ be a minimizer of $V_\epsilon$. Let $y^\star\in\mathbb{R}^m$ be a minimizer of the unperturbed value function $V$. Then, we have
\begin{equation}\nonumber
|V(y^\star)-V_\epsilon(y^\star_\epsilon)| \leq C\epsilon
\end{equation}
for a constant $C$ depending only on $f$.
\end{theorem}

\begin{proof}
Suppose that instead
\begin{equation}\nonumber
V(y^\star)-V_\epsilon(y^\star_\epsilon) > C\epsilon.
\end{equation}
Combining this with the minimality of $y^\star$, we have
\begin{equation}\nonumber
V_\epsilon(y^\star_\epsilon) + C\epsilon < V(y^\star) \leq V(y^\star_\epsilon),
\end{equation}
which contradicts Proposition~\ref{Prop:SemiConcavePerturbedValueEstimate}. A similar argument using the minimality of $y^\star_\epsilon$ is contradictory to the possibility that
\begin{equation}\nonumber
V_\epsilon(y^\star_\epsilon)-V(y^\star) > C\epsilon.
\end{equation}
Then, the result holds.
\end{proof}
Note that the constant $C$ is equal to $\frac{1}{2} L^2$.
Therefore, once we obtain an approximation $y_\epsilon^\star$, 
we can explicitly compute a bound, $C\epsilon$, of the gap between the approximately optimal objective value $V_\epsilon (y_\epsilon^\star)$ and the optimal value $V (y^\star)$.
This bound is useful to determine an appropriate resolution $\epsilon$ of the semiconcave approximation.
Furthermore, the bound is linear in $\epsilon$ and hence the suboptimality gap converges to zero as $\epsilon \to 0$ (see 
Figure~\ref{Fig:Convergence} in Section~\ref{result}
for a numerical experiment result).

\subsection{Uniformly Parabolic Approximation}\label{uni_app}

Relaxing the uniform parabolicity assumption (Assumption~\ref{Assumption:UniformEllipticity}) is particularly important when $0 \in \mathbb{A}$.
To relax Assumption~\ref{Assumption:UniformEllipticity}, we consider an approximation scheme for the case that the dynamics of $X$ do not satisfy the assumption. 
The key idea  is to perturb the dynamics of $X$ with extra sources of risk (or uncertainty) so that uniform parabolicity holds. We then show the resulting perturbed value function converges to that of the unperturbed problem as $\eta \to 0$.

Let $\hat W$ be an independent $n$-dimensional Brownian motion, and let $\hat{\mathcal{F}}$ be the filtration generated by the joint process $(W,\hat W)$. Let $\hat{\mathcal{A}}$ represent the collection of all $\hat{\mathcal{F}}$-adapted processes valued in $\mathbb{A}$. For each control $A\in\hat{\mathcal{A}}$ and $\eta \in\mathbb{R}$, define $\hat{X}^{A,\eta}$ as the unique strong solution of the perturbed SDE
\begin{equation}\nonumber
d\hat{X}^{A,\eta}_t = \mu(\hat{X}^{A,\eta}_t,A_t)\, dt + \sigma(\hat{X}^{A,\eta}_t,A_t)\,dW_t + \eta\,d\hat{W}_t.
\end{equation}
We define the following $\eta$-perturbed value function:
\begin{equation}\nonumber
V_\eta(y) := \inf\limits_{A\in\hat{\mathcal{A}}}\mathbb{E}\left[f(g(\hat{X}^{A,\eta}_T),y)\right].
\end{equation}
For any $\eta > 0$, the dynamics of the perturbed process $\hat{X}^{A, \eta}$ satisfy Assumption~\ref{Assumption:UniformEllipticity}. Then, we can apply the results of Section~\ref{Section:MainResults} to minimize $V_\eta$. The goal of this subsection is to show convergence to the unperturbed case as $\eta\to 0$.

First, we have the following estimate involving the value function under the perturbed dynamics.

\begin{proposition}\label{Prop:EpsilonInequality}
Suppose that $f$ is Lipschitz continuous. Fix $y\in\mathbb{R}^m$. Then, for any $\eta,\eta'\geq 0$, we have the estimate
\begin{equation}\nonumber
|V_\eta(y)-V_{\eta'}(y)| \leq C|\eta-\eta'|
\end{equation}
for some constant $C$ which depends only on $\mu$, $\sigma$, $\mathbb{A}$, $f$, and $g$.
\end{proposition}

\begin{proof}

Fix $y\in\mathbb{R}^m$ and $\eta,\eta'\geq 0$. For any $\delta >0$, let $A\in\hat{\mathcal{A}}$ be a $\delta$-suboptimal control such that
\begin{equation}\label{Eqn:EpsilonEstimate1}
V_\eta(y) + \delta \geq \mathbb{E}\left[f(g(\hat{X}^{A,\eta}_T),y)\right].
\end{equation}
By a standard argument using the Lipschitz constants of the perturbed dynamics, we have
\begin{equation}\nonumber
\mathbb{E}\left[\|\hat{X}^{A,\eta}_T-\hat{X}^{A,\eta'}_T\|^2\right] \leq C(\eta-\eta')^2
\end{equation}
for some constant $C$ which depends only on the Lipschitz constants of $\mu$ and $\sigma$ and the set $\mathbb{A}$ (for example, see \cite{Touzi2013}). Because $f$ and $g$ are  Lipschitz continuous, we have
\begin{eqnarray}\label{Eqn:EpsilonEstimate2}
\mathbb{E}\left[|f(g(\hat{X}^{A,\eta'}_T),y)-f(g(\hat{X}^{A,\eta}_T),y)|\right] & \leq & C\mathbb{E}\left[\|\hat{X}^{A,\eta'}_T-\hat{X}^{A,\eta}_T\|\right] \nonumber\\
& \leq & C\left(\mathbb{E}\left[\|\hat{X}^{A,\eta'}_T-\hat{X}^{A,\eta}_T\|^2\right]\right)^{1/2}\nonumber\\
& \leq & C|\eta-\eta'|
\end{eqnarray}
for a new constant $C$. Combining \eqref{Eqn:EpsilonEstimate1} and \eqref{Eqn:EpsilonEstimate2}, we conclude
\begin{equation}\nonumber
V_\eta(y) + \delta \geq \mathbb{E}\left[f(g(\hat{X}^{A,\eta'}_T),y)\right] - C|\eta-\eta'| \geq V_{\eta'}(y) - C|\eta-\eta'|.
\end{equation}
Because $\delta$ was arbitrary, and by symmetry between $\eta$ and $\eta'$, the desired result follows.

\end{proof}

\begin{remark}
In this proof we crucially assume that $f$ is Lipschitz continuous. Therefore, if we are to apply this approximation scheme to mean-variance optimization, we would need to attempt to modify the proof.
\end{remark}

Next, we  prove the intuitive statement that enlarging the filtration $\mathbb{F}$ to $\hat{\mathbb{F}}$ does not affect the value function when $\eta = 0$.

\begin{proposition}\label{Prop:EpsilonInequality2}
For any $y\in\mathbb{R}^m$, we have $V_0(y) = V(y)$.\end{proposition}

\begin{proof}
Fix $y\in\mathbb{R}^m$ and consider the following two functions:
\begin{equation}\nonumber
\begin{split}
u(t,x) & :=  \inf\limits_{A\in\mathcal{A}}\mathbb{E}_t\left[f(g(X^A_T),y)\mid X^A_t=x\right] \\
\hat{u}(t,x) & :=  \inf\limits_{A\in\hat{\mathcal{A}}}\mathbb{E}_t\left[f(g(\hat{X}^A_T),y)\mid\hat{X}^A_t=x\right].\nonumber
\end{split}
\end{equation}
By Proposition~\ref{Prop:ViscositySolution}, both $u$ and $\hat{u}$ are the viscosity solutions of the same HJB equation. Furthermore,
\begin{equation}\nonumber
V(y) = u(0,x)\text{ and }V_0(y) = \hat{u}(0,x).
\end{equation}
Then, the result follows by the uniqueness of viscosity solutions of \eqref{Eqn:ValueHJB}.
\end{proof}

Finally, combining these two results, we show a suboptimality bound when using this uniformly parabolic approximation.

\begin{theorem}[Convergence and suboptimality bound II]\label{thm:bound1}
Suppose that $f$ is Lipschitz continuous. Fix $\eta >0$ and let $y^\star_\eta\in\mathbb{R}^m$ be a minimizer of $V_\eta$. Let $y^\star\in\mathbb{R}^m$ be a minimizer of the unperturbed value function $V$. Then, we have
\begin{equation}\nonumber
|V(y^\star)-V_\eta(y^\star_\eta)| \leq C\eta
\end{equation}
for some constant $C$ which depends only on $\mu$, $\sigma$, $\mathbb{A}$, $f$, and $g$.
\end{theorem}

\begin{proof}
Suppose that 
\begin{equation}\nonumber
V(y^\star)-V_\eta(y^\star_\eta) > C\epsilon.
\end{equation}
Note that $y^\star$ is a minimizer, we have
\begin{equation}\nonumber
V_\eta(y^\star_\eta) + C\eta < V(y^\star) = V_0(y^\star) \leq V_0(y^\star_\eta),
\end{equation}
which is contradictory to Proposition~\ref{Prop:EpsilonInequality}. A similar argument using the minimality of $y^\star_\eta$ contradicts the possibility that
\begin{equation}\nonumber
V_\eta(y^\star_\eta)-V(y^\star) > C\eta.
\end{equation}
Then, the result follows.
\end{proof}

Note that we can obtain a solution to the original problem by choosing $\eta>0$ sufficiently small, depending on the constant $C$, and applying the gradient descent algorithm to minimize $V_\eta$. In practice, the constant $C$ can be computed explicitly in terms of the size of $\mathbb{A}$ and the Lipschitz constants on $\mu$, $\sigma$, $f$ and $g$ by following the steps in the proofs above. Consequently, 
using Theorem~\ref{thm:bound1}, we can explicitly compute a bound on the suboptimality gap between the obtained solution and the global minimizer $y^\star$. As $\eta \to 0$, the suboptimality gap tends to zero.

\subsection{Nonconvex Case}\label{Subsection:Nonconcave}

In the previous two subsections, we showed convergent approximation schemes when either the uniform semiconcavity or the uniform parabolicity assumption does not hold. In the case of the convexity assumption (Assumption~\ref{Assumption:Convexity}), we do not provide a direct approximation scheme, but instead note analogies of the results in Section~\ref{Section:MainResults} which lead to a proximal supergradient descent algorithm.

In particular, we show that in the absence of Assumption~\ref{Assumption:Convexity}, all results in Section~\ref{Section:MainResults} hold with the caveat that $V$ is semiconcave instead of convex.

\begin{corollary}
Suppose that, for any fixed $y\in\mathbb{R}^m$, there exists $A\in\mathcal{A}$ such that
\begin{equation}\nonumber
V(y) = \mathbb{E}\left[f(g(X^A_T),y)\right].
\end{equation}
If we denote by $\partial V(y)$ the collection of proximal supergradients of $V$ at $y$, we have
\begin{equation}\nonumber
\mathbb{E}\left[D_yf(g(X^A_T),y)\right] \in \partial V(y).
\end{equation}
\end{corollary}

This follows immediately from the inequality proved in Proposition~\ref{Prop:Semiconcavity}. Furthermore, because of the PDE results of Section~\ref{Section:PDE}, we know that there exists an optimal control at each $y\in\mathbb{R}^m$. We make the following conclusion:	

\begin{corollary}
The function $V$ is semiconcave. For each $y\in\mathbb{R}^m$, we can compute $V(y)$ and a proximal supergradient at $y$ by solving PDEs \eqref{Eqn:ValueHJB} and \eqref{Eqn:GradientPDE}, respectively.
\end{corollary}

Consequently, the globally convergent gradient descent algorithm may be replaced with a proximal supergradient descent algorithm which converges to a local minimum. For more on tools from non-smooth analysis and semiconcavity, see \cite{Cannarsa2004,Katzourakis2015}.

\subsection{Proof of Theorem~\ref{Theorem:Summary}}\label{Subsection:ProofMainTheorem}

{{}At this point, the proof of Theorem~\ref{Theorem:Summary} is just a matter of piecing together the various results from the previous two sections. We provide a brief outline of where each step is used in the following:

\begin{proof}[Proof of Theorem~\ref{Theorem:Summary}]
Recall the statement of the theorem is broken into three pieces. We follow this separation in this proof.

\begin{enumerate}
\item The first section of the theorem is Proposition~\ref{Thm:BilevelOptimization}.

\item Applying the results of Sections~\ref{semi_app}--\ref{Subsection:Nonconcave}, it follows that the approximate problem satisfies Assumptions~\ref{Assumption:UniformSemiconcavity} and \ref{Assumption:UniformEllipticity}. Assumption~\ref{Assumption:UniformSemiconcavity} follows from Proposition~\ref{Prop:InfConvolutionProperties}, while Assumption~\ref{Assumption:UniformEllipticity} follows from directly from the dynamics of the approximate problem. Then the results of the second section of the theorem follow from the discussion in Section~\ref{Subsection:Nonconcave}.

\item Lastly, we consider the convergence result. By combining Theorem~\ref{Thm:bound2} and Theorem~\ref{thm:bound1}, we immediately conclude that
\[|V(y^\star) - V_{\epsilon,\eta}(y^\star_{\epsilon,\eta})| \leq C(\epsilon+\eta)\]
for some $C>0$ which depends only on the Lipschitz constants of $f$, $g$, $\mu$, and $\sigma$. Similarly, applying Propositions~\ref{Prop:SemiConcavePerturbedValueEstimate}--\ref{Prop:EpsilonInequality2}, we conclude that
\[V_{\epsilon,\eta}(y^\star_{\epsilon,\eta})\leq V_{\epsilon,\eta}(y^\star) + C\left(\epsilon+\eta\right)\]
for some (possibly different) $C>0$, which also depends only on the Lipschitz constants above. Lastly, combining these two inequalities, we conclude the stated convergence result.
\end{enumerate}
\end{proof}
}

\section{Example: Mean-CVaR Portfolio Optimization}\label{Section:Example}

In this section, we illustrate a practical use of our main results and approximation methods in an application to portfolio optimization under a mean-CVaR objective. Our goal is to ultimately use this methodology to compute the efficient frontier representing the trade-off between maximizing expected log-return and minimizing the CVaR of losses. We emphasize that dynamic optimization can significantly reduce CVaR while maintaining the same expected return as compared to optimal static investment strategies.

\subsection{Problem Formulation}

Consider a market consisting of $n$ risky assets evolving via the SDE
\begin{equation} \nonumber
\frac{dS^{(i)}_t}{S^{(i)}_t} = \mu_i \,dt + \Sigma^{1/2}_{i,j}\,dW^{(j)}_t
\end{equation}
for each $i\in\{1,\ldots ,n\}$ and $j \in \{1, \ldots, d\}$. Here $\mu\in\mathbb{R}^n$ is a vector of drifts and $\Sigma$ is the covariance matrix of returns. We also assume that there exists a risk-free asset with drift $r$.

We assume that we choose a control $A$, which is a progressively-measurable process lying in some compact set $\mathbb{A}$, representing the percent of the portfolio exposed to each of the $n$ risky assets. For example, we might choose
$\mathbb{A} := \left\{a\in\mathbb{R}^n\mid a^\top\Sigma a\leq l \right\}$
for a constant $l$ corresponding to a hard portfolio risk cap.

With this setup, our portfolio value $Z$ evolves via the SDE
\begin{equation}\nonumber
\frac{dZ^A_t}{Z^A_t} = \left[r+A_t^\top(\mu-r\,1)\right]\, dt + A_t^\top \Sigma^{1/2}\,dW_t.
\end{equation}
For simplicity, we consider the log value of the portfolio, $X^A_t := \log Z^A_t$, which can be seen to solve
\begin{equation}\nonumber
dX^A_t = \left[r + A_t^\top(\mu-r\,1)-\frac{1}{2}A_t^\top\Sigma A_t\right]\,dt + A_t^\top\Sigma^{1/2}\,dW_t.
\end{equation}
Without loss of generality, we assume $Z^A_0 = S_0 = 1$. Then, $X^A_0 = 0$ and we can interpret $X^A_t$ as the log-returns of the portfolio up to time $t$.

In this section, we consider the problem of minimizing a mean-CVaR objective:
\begin{equation}\label{Eqn:ExampleProblem}
\inf\limits_{A\in\mathcal{A}}\left[\mathbb{E}\left[-X^A_T\right]+\lambda\text{CVaR}_\alpha\left[-X^A_T\right]\right]
\end{equation}
for fixed $\lambda >0$ and $\alpha\in(0,1)$. By varying $\lambda$, we can compute a subset of the efficient frontier between expected log-return and the CVaR of losses.

\subsection{Solution via Gradient Descent}\label{Subection:NumericalApprox}

From the results of Section~\ref{Section:MainResults},  the problem \eqref{Eqn:ExampleProblem} is equivalent to the bilevel optimization
\begin{equation}\nonumber
\inf\limits_{y\in\mathbb{R}} V(y),
\end{equation}
where
\begin{equation}\nonumber
V(y) := \inf\limits_{A\in\mathcal{A}}\mathbb{E}\left[f(g(X^A_T),y)\right].
\end{equation}
Here, we take $g(X_T^A) = -X_T^A$ and
\begin{equation}\nonumber
f(x,y) := x + \lambda\left(y+\frac{1}{1-\alpha}\left(x-y\right)^+\right).
\end{equation}
We can check that $A\mapsto g(X^A_T)$ is convex almost surely and $f$ is convex and non-decreasing in $x$. Therefore, Assumption~\ref{Assumption:Convexity} is satisfied, i.e., $(A, y) \mapsto f(g(X_T^A),y)$ is jointly convex. However, both Assumption~\ref{Assumption:UniformSemiconcavity} and Assumption~\ref{Assumption:UniformEllipticity} are potentially violated. Therefore, we need to apply the approximations of Section~\ref{Section:Approximation} before we can proceed with the proposed gradient descent method.

\begin{theorem}\label{Thm:ApproximationForExample}
For any $\epsilon >0$, we have the inf-convolution of $f (g(x), y)$ is
\begin{equation}\nonumber
f_\epsilon(g(x),y) := \left\{\begin{array}{ll}
-x + \lambda\left(y - \frac{x+y}{1-\alpha}\right)-\frac{1}{2}\left(\frac{\alpha}{1-\alpha}\lambda\right)^2\epsilon & \text{for }x+y<-\frac{\alpha}{1-\alpha}\lambda\epsilon\\
\frac{1}{2\epsilon}\left(y+x\right)^2-(1+\lambda)x& \text{for }-\frac{\alpha}{1-\alpha}\lambda\epsilon\leq x+y\leq \lambda\epsilon \\
-x + \lambda y-\frac{1}{2}\lambda^2\epsilon & \text{for }x+y>\lambda\epsilon.
\end{array}\right.
\end{equation}
If we also consider the perturbed dynamics
\begin{equation}\nonumber
d\hat{X}^{A,\epsilon}_t = \left[r+A^\top_t\left(\mu-r\,1\right)-\frac{1}{2}A_t^\top\Sigma A_t\right]\,dt + A_t^\top\Sigma^{1/2}\,dW_t + \epsilon\,d\hat{W}_t
\end{equation}
and the perturbed value function
\begin{equation}\nonumber
V_\epsilon(y) := \inf\limits_{A\in\hat{\mathcal{A}}}\mathbb{E}\left[f_\epsilon\left(g(\hat{X}^{A,\epsilon}_T),y\right)\right],
\end{equation}
then Assumptions~\ref{Assumption:UniformSemiconcavity}--\ref{Assumption:Convexity} all hold and we can apply the gradient descent method from Section \ref{Section:MainResults} to minimize $V_\epsilon$. Furthermore, there exists a constant $C$ which depends only on $\mu$, $r$, $\Sigma$, $\lambda$, $\alpha$, and $\mathbb{A}$ such that
\begin{equation}\nonumber
|V(y^\star)-V_\epsilon(y^\star_\epsilon)|\leq C\epsilon
\end{equation}
for any $y^\star_\epsilon\in\mathbb{R}^m$ which minimizes $V_\epsilon$ and $y^\star\in\mathbb{R}^m$ which minimizes $V$.
\end{theorem}

The proof is elementary but tedious. Verifying the proposed expression for $f_\epsilon$ is an exercise in minimizing piecewise functions. We check that $f_\epsilon$ satisfies the Assumption~\ref{Assumption:UniformSemiconcavity}, the new dynamics clearly satisfy Assumption~\ref{Assumption:UniformEllipticity}, and together they satisfy Assumption~\ref{Assumption:Convexity} by checking the sufficient conditions in Example~\ref{Example:SufficientConditionsConvexity}. Finally, the error bounds result from combining the error bounds from Sections~\ref{semi_app} and~\ref{uni_app}.

\begin{remark}
In this theorem we use a single parameter, $\epsilon$, for both approximation schemes. In practice, we can modify this result to use separate parameters for each scheme. This could be useful for separating out the numerical error due to each approximation.
\end{remark}

\subsection{Numerical Results: Computation of an Efficient Frontier}
\label{result}

In this section we consider a concrete example involving selection between a single risky asset, representing a US stock index, and a risk-free asset. We compute an efficient frontier representing the trade-off between expected log-return and CVaR when using optimal dynamic strategies. For comparison, we compare to an efficient frontier when restricting to static strategies, i.e. strategies where $A$ is constant over time, representing a fixed leverage ratio.

\begin{figure}[tb]
\centering
\includegraphics[height=2in]{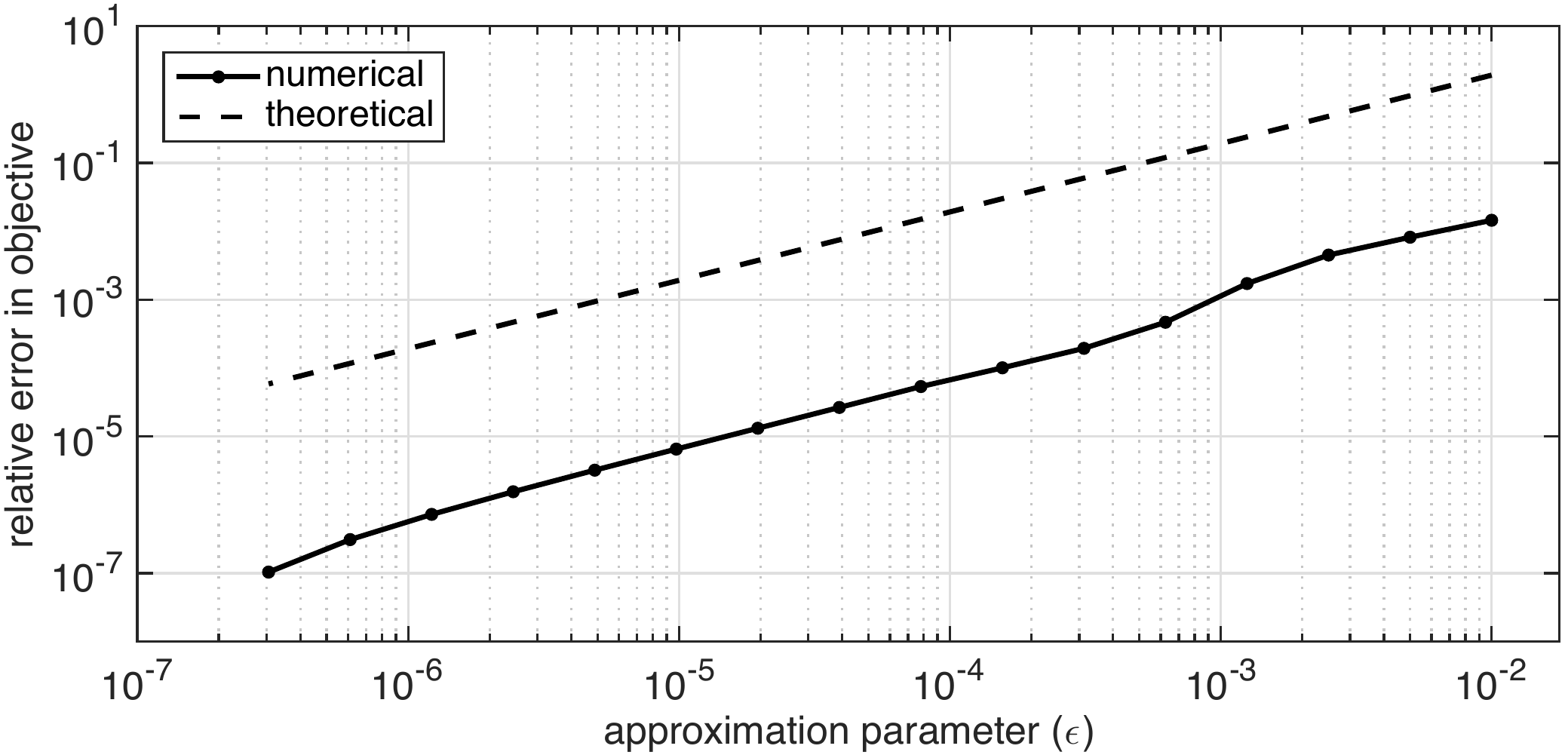}
\caption{
The relative error $|V(y^\star) - V_\epsilon (y_\epsilon^\star)|$ in the objective.
We compare a numerical error (solid) and 
 an explicit calculation of the theoretical error bounds (dashed) provided in Section~\ref{Section:Approximation}.}
\label{Fig:Convergence}
\end{figure}

For our example, we choose $\mu=11\%$, $\sigma=20\%$\footnote{This choice corresponds, roughly, to the historical arithmetic mean and standard deviation of annual returns on the S\&P 500, including dividend reinvestment, over the period 1928--2014. However, we emphasize that the exact choice of parameters should not be taken too seriously in this example.}, and $r=1\%$ as market parameters. We take our time horizon as $T=1$ and constrain our leverage ratio to lie within the range $\mathbb{A} := [-6,+6].$\footnote{We choose this range to correspond, roughly, to the maximum leverage a qualifying US investor can achieve with a portfolio margin policy, as described at http://www.finra.org/industry/portfolio-margin-faq. In practice, the exact constraints depend upon the type of investor and financial instruments used for investment. We emphasize that this choice is meant for illustration only.} Finally, we consider CVaR at the $\alpha=95\%$ threshold.

For each fixed $\lambda > 0$, we solve the corresponding dynamic mean-CVaR optimization problem using the techniques outlined in Section~\ref{Subection:NumericalApprox}.\footnote{{{} We use finite-difference solvers for PDEs~\eqref{Eqn:ValueHJB} and \eqref{Eqn:GradientPDE}. We apply upwinding methods to obtain a monotone scheme when solving the gradient PDE \cite{Courant1952,Barles1991}. It is difficult to make any quantitative statements regarding the choice of discretization parameters. However, we note that the approximation parameter, $\epsilon$, should not be too small compared to the mesh-spacing used in the finite-difference method or the difference between $f$ and $f_\epsilon$ will become numerically irresolvable.}} An example of the convergence of the approximation scheme described in Theorem~\ref{Thm:ApproximationForExample}, compared with the corresponding theoretical error bounds, is shown in Figure~\ref{Fig:Convergence}.
This illustrates the linear decrease in relative error from approximation as a function of the approximation parameter $\epsilon$. 

\begin{figure}[tb]
\centering
\includegraphics[height=2in]{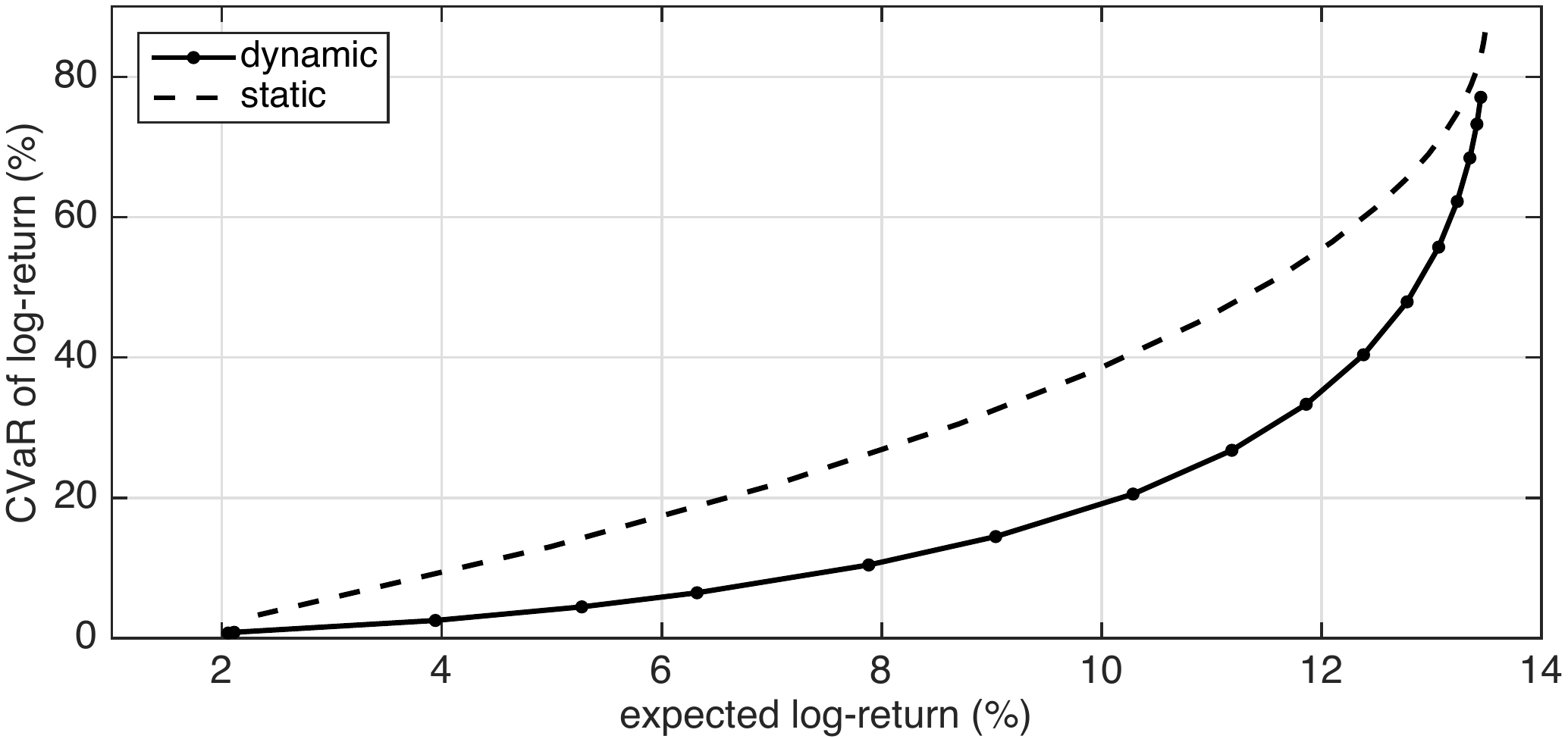}
\caption{The efficient frontier of Mean-CVaR portfolio optimization, representing the possible trade-off between maximizing expected log-returns and minimizing CVaR, computed by varying $\lambda \in (0,1]$.}
\label{Fig:EfficientFrontier}
\end{figure}

We compute points on the efficient frontier between expected log-return and CVaR by varying $\lambda$ over the interval $(0,1]$. At each value of $\lambda$, we compute the expected log-return under the optimal control by solving a linear parabolic equation analogous to \eqref{Eqn:GradientPDE}. We then compute the corresponding CVaR using our computation of expected log-return and the mean-CVaR objective.\footnote{We emphasize that one might alternatively choose to run Monte Carlo simulations using the optimal feedback control $a^\star(t,x)$ to obtain estimates of the expected log-return and CVaR. For more on this scalarization methodology to compute points on the efficient frontier, see \cite{Boyd2004}.}
The resulting frontier is shown in Figure~\ref{Fig:EfficientFrontier} (solid).

For comparison, we consider the same optimization problem when restricted to a subcollection of static controls, defined as
\begin{equation} \nonumber
\mathcal{A}_{\text{static}} := \left\{A\in\mathcal{A}\mid\exists a\in\mathbb{A}\text{ such that }A(t)=a\text{ for all }t\in[0,T]\text{ a.s.}\right\}.
\end{equation}
These strategies represent constant leverage portfolios. An important example of these is the ``buy and hold'' strategy, e.g. $A(t)\equiv 1$. Under this class of controls, $X^A_T$ is normally-distributed. Therefore, we can directly compute optimal strategies and construct the efficient frontier.

In Figure~\ref{Fig:EfficientFrontier}, we illustrate a comparison between the efficient frontier under our dynamic strategies and under static strategies. We see that by employing strategies with dynamic leverage, we can significantly reduce CVaR at the 95\% quantile while maintaining the same expected log-return, as compared to a static leverage strategy. Similarly, we can increase expected log-return while maintaining the same CVaR using a dynamic strategy. For example, the static buy-and-hold strategy,  $A(t) \equiv 1$, has an expected log-return of 9\% and CVaR of approximately 32\%. By employing strategies with dynamic leverage, we can reduce CVaR by approximately 50\% while maintaining the same expected log-return, or alternatively increase expected log-return by approximately 30\% while maintaining the same CVaR.

\begin{figure}[tb]
\centering
\includegraphics[height=2in]{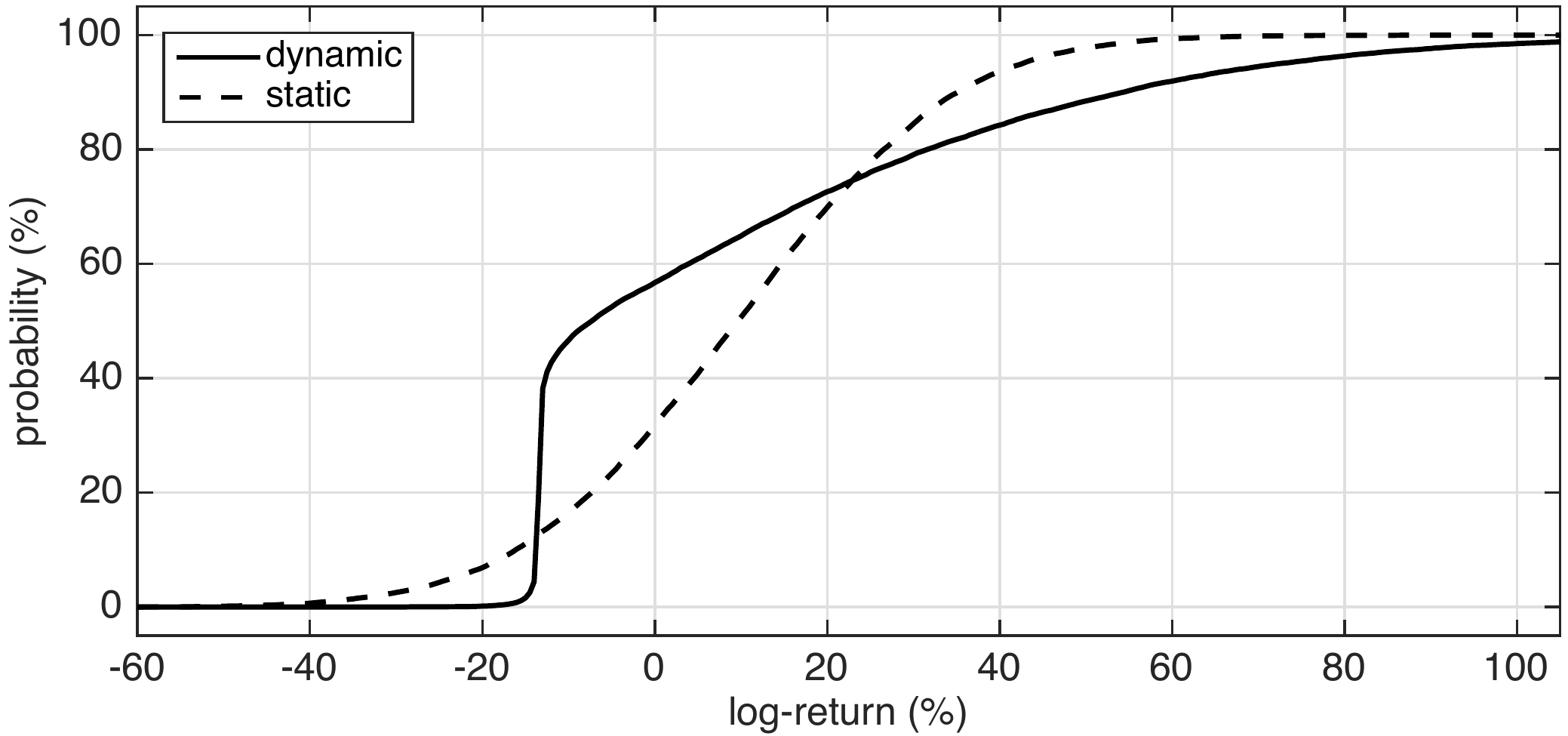}
\caption{The cumulative distribution function of $X^A_T$ when following the static buy-and-hold strategy and the optimal dynamic strategy which achieves the same expected log-return.}
\label{Fig:Distribution}
\end{figure}

We next turn our attention to an examination of statistical and qualitative properties of the optimal dynamic control and resulting returns. In Figure~\ref{Fig:Distribution}, we illustrate the cumulative distribution function (CDF) of $X^A_T$ under the optimal dynamic control corresponding an expected log-return of 9\%. We compare this to the CDF of $X^A_T$ under the buy-and-hold strategy, which follows a normal distribution. While both of these distributions have the same expected value, the one corresponding to the optimal dynamic strategy has significantly fatter (right) tails on the upside and an effective (left) floor on losses on the downside. We attribute this to a (de-)leveraging effect of the dynamic strategy whereby it increases leverage significantly once it has ``locked in'' gains and will de-leverage only as needed to discourage losses exceeding a certain threshold.

\begin{figure}[tb]
\centering
\includegraphics[height=3in]{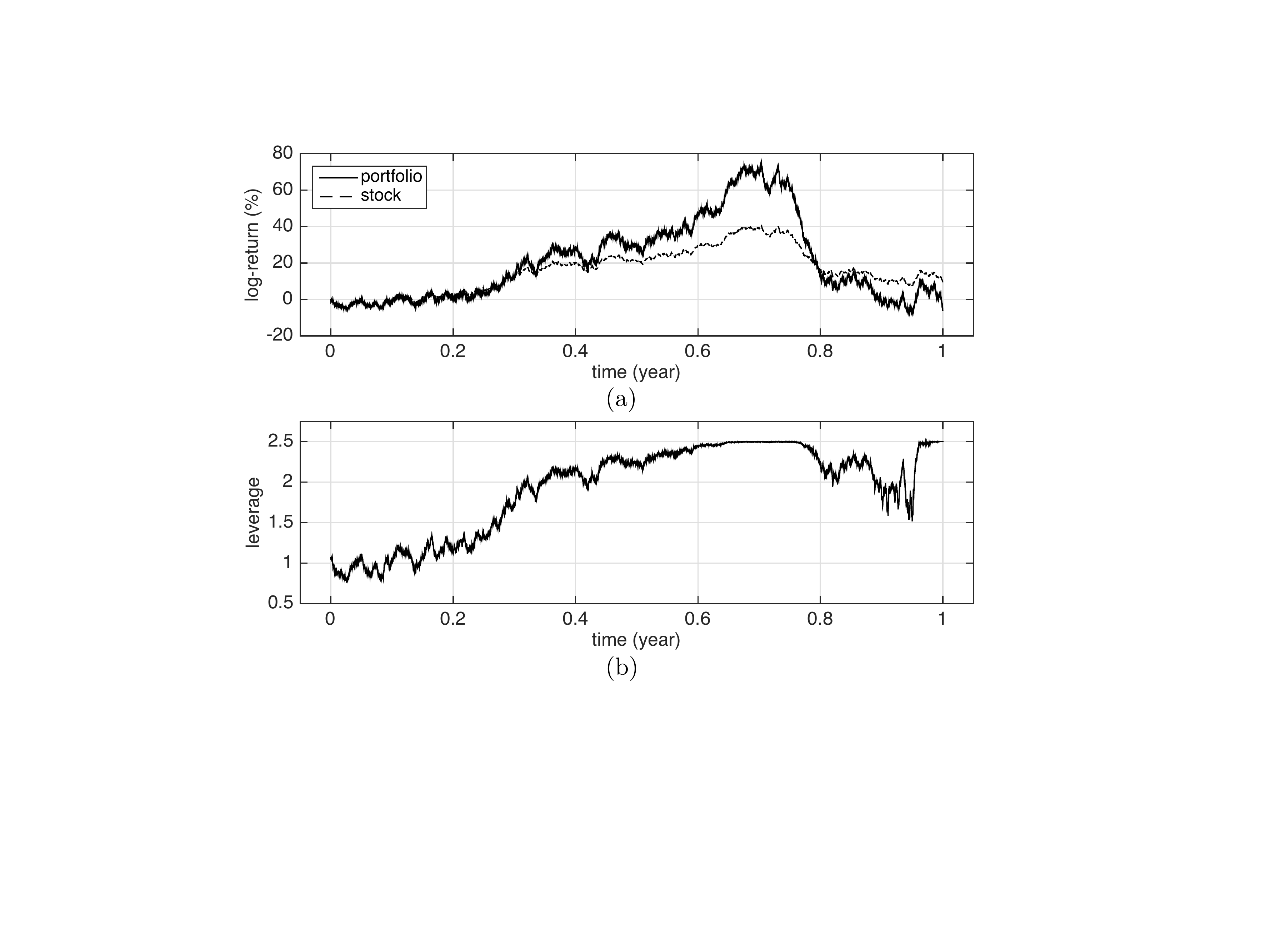}
\caption{(a) A sample path of stock prices and the corresponding portfolio log-return process ($X^{A^\star}$), and (b) the corresponding optimal leverage process ($A^\star$).}
\label{Fig:Pathwise}
\end{figure}

\begin{figure}[tb]
\centering
\includegraphics[height=3in]{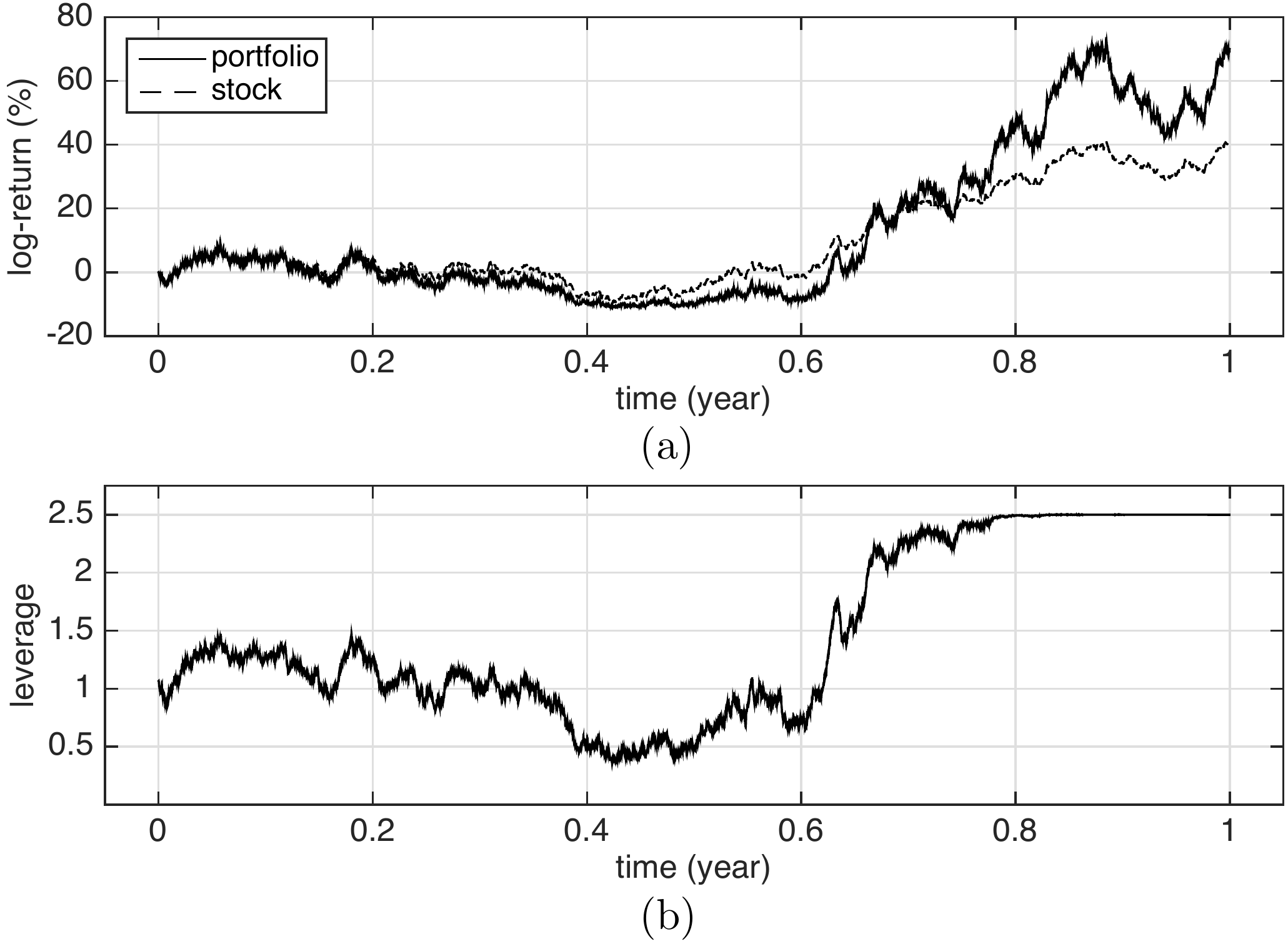}
\caption{(a) A sample path of stock prices and the corresponding portfolio log-return process ($X^{A^\star}$), and (b) the corresponding optimal leverage process ($A^\star$).}
\label{Fig:Pathwise2}
\end{figure}

This qualitative tendency of the optimal strategy to increase in leverage once it has locked in gains is emphasized further by sample paths illustrated in Figure~\ref{Fig:Pathwise}. Here, we illustrate a particular sample path of stock prices (quoted as log-return), as well as the corresponding optimal dynamic leverage process, $A^\star$, and the resulting portfolio log-return process, $X^{A^\star}$. Note that the stock price corresponds to the log-return under the static buy-and-hold strategy, $A(t) \equiv 1$. We observe that early on in the period, the leverage process increases or decreases in sync with overall portfolio returns. However, as it becomes later in the period and the portfolio return is positive, the optimal leverage increases significantly before being capped at a fixed value. The optimal strategy generally does not appear to decrease leverage late in the period, even with stock price declines, unless it is risking falling below the loss threshold seen in the jump in Figure~\ref{Fig:Distribution}.

In Figure~\ref{Fig:Pathwise2}, we illustrate an alternative sample path which emphasizes how the increasing leverage can lead to large returns on the upside. In this path, the leverage process, $A^\star$, initially decreases to lower risk as the portfolio takes initial losses. However, in the latter half of the period, as stock prices rise, the increasing leverage leads to a return on the portfolio which significantly exceeds that of the buy-and-hold strategy. It is this transition from low leverage when avoiding tail losses to high leverage when locking in gains which allows the strategy to maintain a low CVaR while maximizing expected log-return.

The tendency of the optimal dynamic strategy to keep leverage higher than the static strategy unless it is facing losses also helps explain the skew seen in Figure~\ref{Fig:Distribution}. Because the dynamic strategy has the option to decrease its leverage to stop losses, it can achieve a significantly lower CVaR while maintaining a preference for high leverage, which contributes to large returns in positive outcomes. However, there is no such thing as a free lunch; in neutral outcomes, the positive correlation between log-returns and leverage leads to decay in portfolio value from convexity \cite{Perold1988}. In this sense, the optimal dynamic strategy shares many qualitative features with Constant Proportion Portfolio Insurance (CPPI) strategies \cite{Black1992}. This makes sense as CPPI strategies are generally employed to limit downside losses, while maintaining upside gains, using dynamic trading.

In this example, we chose to consider only a single risky asset for ease of interpretation of the strategy. However, it is clear from the generality of Section~\ref{Subection:NumericalApprox} that we could perform the same computations with multiple assets without increasing the size of the state space in the stochastic optimal control sub-problem. In the multiple asset case, the extra complexity appears in the need to solve a quadratic program with more decision variables
 at each step in the solution of the HJB in~\eqref{Eqn:ValueHJB}.

\appendix
\section{Coherent Extremal Risk Measures}\label{app:1}

\begin{proposition}[Coherence]\label{Prop:RiskMeasure}
Suppose the following properties of $f:\mathbb{R}\times\mathbb{R}^m\to\mathbb{R}$ hold:
\begin{itemize}
\item Positive-homogeneity and normalization: $f(ax,ay) = af(x,y)$ for $a > 0$ and $\inf\limits_{y\in\mathbb{R}^m} f(0,y) = 0$;
\item Monotonicity: $x\mapsto f(x,y)$ is non-decreasing for each $y\in\mathbb{R}^m$;
\item Sub-additivity: $f(x_1 + x_2, y_1 + y_2) \leq f(x_1, y_1) + f(x_2, y_2)$ for $x_1,x_2 \in \mathbb{R}$ and $y_1,y_2 \in \mathbb{R}^m$;
and
\item Translation: For every $a\in\mathbb{R}$, there exists an invertible function $\phi:\mathbb{R}^m\to\mathbb{R}^m$ such that for every $(x,y)\in\mathbb{R}\times\mathbb{R}^m$ we have $f(x+a,y) = f(x,\phi(y)) + a$.
\end{itemize}
Then, the function $\rho:L^2(\Omega)\to\mathbb{R}$ defined as
\begin{equation}
\rho(\xi) := \inf\limits_{y\in\mathbb{R}^m}\mathbb{E}\left[f(\xi,y)\right]
\end{equation}
is a coherent (extremal) risk measure.
\end{proposition}
\begin{proof}
Recall, the four properties of a coherent risk measure  are positive homogeneity, monotonicity, sub-additivity and translational invariance \cite{Artzner1999}. We show below how each of these comes from the assumptions on $f$.
\begin{enumerate}
\item Positive-homogeneity: We compute directly
\begin{equation}\nonumber
\rho(a\xi) = \inf\limits_{y\in\mathbb{R}^m}\mathbb{E}\left[f(a\xi,y)\right]
= \inf\limits_{ay\in\mathbb{R}^m}\mathbb{E}\left[f(a\xi,ay)\right]
 = a \inf\limits_{y\in\mathbb{R}^m} \mathbb{E} \left [f(\xi,y)  \right ]= a \rho(\xi)
\end{equation}
for $a > 0$. When $a = 0$, 
\begin{equation}\nonumber
\rho(0) = \inf\limits_{y\in\mathbb{R}^m}\mathbb{E}\left[f(0,y)\right] = \inf\limits_{y\in\mathbb{R}^m}f(0,y) = 0.
\end{equation}

\item Monotonicity: Let $\xi,\xi'\in L^2(\Omega)$ such that $\xi \leq \xi'$ almost surely. Then, 
\begin{equation}\nonumber
\rho(\xi) = \inf\limits_{y\in\mathbb{R}^m}\mathbb{E}\left[f(\xi,y)\right]\leq\inf\limits_{y\in\mathbb{R}^m}\mathbb{E}\left[f(\xi',y)\right] = \rho(\xi').
\end{equation}

\item Sub-additivity: Fix $\xi_1, \xi_2 \in L^2(\Omega)$. We compute
\begin{equation}\nonumber
\begin{split}
\rho(\xi_1 + \xi_2) = \inf_{y \in \mathbb{R}^m}
\mathbb{E} [f(\xi_1+ \xi_2, y)] &= 
\inf_{y_1, y_2 \in \mathbb{R}^m}
\mathbb{E} [f(\xi_1+ \xi_2, y_1 + y_2)]\\
&\leq 
\inf_{y_1, y_2 \in \mathbb{R}^m}
\mathbb{E} [f(\xi_1, y_1) + f(\xi_2, y_2)]
= \rho(\xi_1) + \rho(\xi_2).
\end{split}
\end{equation}

\item Translational-invariance: Let $\xi\in L^2(\Omega)$ and $a\in\mathbb{R}$. Then
\begin{equation}\nonumber
\rho(\xi+a) = \inf\limits_{y\in\mathbb{R}^m}\mathbb{E}\left[f(\xi+a,y)\right] = \inf\limits_{y\in\mathbb{R}^m}\mathbb{E}\left[f(\xi,\phi(y))+a\right] = \rho(\xi)+a.
\end{equation}
\end{enumerate}
\end{proof}

{{}\section{Proof of Theorem \ref{Thm:Convexity}}

\begin{proof}
Let $y,y'\in\mathbb{R}^m$ and $\theta \in[0,1]$. For any $\epsilon >0$, let $A,A'\in\mathcal{A}$ be $\epsilon$-suboptimal controls such that
\begin{equation}\nonumber
V(y) + \epsilon \geq \mathbb{E}\left[f(g(X^A_T),y)\right]\text{ and }V(y') + \epsilon \geq \mathbb{E}\left[f(g(X^{A'}_T),y')\right].
\end{equation}
By Assumption~\ref{Assumption:Convexity}, we obtain
\begin{eqnarray}\nonumber
V(\theta y+(1-\theta)y') & \leq & \mathbb{E}\left[f(g(X^{\theta A+(1-\theta)A'}_T),\theta y+(1-\theta)y')\right] \\\nonumber
& \leq & \mathbb{E}\left[\theta f(g(X^A_T),y)+(1-\theta)f(g(X^{A'}_T),y')\right]\\\nonumber
& \leq & \theta V(y) + (1-\theta)V(y') + \epsilon.\nonumber
\end{eqnarray}
Because $\epsilon$ was arbitrary, the convexity of $V$ follows.
\end{proof}
}

\section{Proof of Proposition~\ref{Prop:InfConvPreseveresConvexity}}\label{app:Prop7}
\begin{proof}
Recall that Assumption~\ref{Assumption:Convexity} states that the map
\begin{equation}\nonumber
(A,y) \mapsto f(g(X^A_T),y)
\end{equation}
is jointly convex, almost surely. We observe that the map
\begin{equation}\nonumber
(A,y,z) \mapsto f(g(X^A_T),z)+\frac{\|y-z\|^2}{2\epsilon}
\end{equation}
is almost surely jointly convex as a sum of convex functions. Therefore,
\begin{equation}\nonumber
(A,y)\mapsto f_\epsilon(g(X^A_T),y) = \inf\limits_{z\in\mathbb{R}^m}\left[f(g(X^A_T),z)+\frac{\|y-z\|^2}{2\epsilon}\right]
\end{equation}
is almost surely jointly convex by the same type of argument as in the proof of Theorem~\ref{Thm:Convexity}.
\end{proof}

\bibliographystyle{siam}

\bibliography{reference}

\end{document}